\pdfoutput=1
\newcommand{\R}{\mathbb{R}}
\renewcommand{\H}{\mathcal{H}}

\newcommand{\ip}[2]{\ensuremath{\left<#1,#2\right>}} 
\newcommand{\norm}[1]{\ensuremath{\left|\left|#1\right|\right|}}
\newcommand{\nt}{\notag}

\renewcommand{\L}{\mathcal{L}}
\renewcommand{\l}{\ell}

\documentclass{article}
\usepackage{authblk,graphicx,url}

\usepackage{graphicx}
\usepackage{amssymb,amsmath}
\usepackage{upgreek}
\usepackage{amsthm}
\usepackage{color}

\newtheorem{theorem}{Theorem}[section]
\newtheorem{lemma}[theorem]{Lemma}
\DeclareMathOperator*{\argmin}{argmin}
\begin{document}
\title{Wavelets on Graphs via Spectral Graph Theory}



\author[*,a,1]{David K Hammond}
\author[b,2]{Pierre Vandergheynst}
\author[c]{R\'emi Gribonval}
\affil[a]{NeuroInformatics Center, University of Oregon, Eugene, USA}
\affil[b]{Ecole Polytechnique F\'ed\'erale de Lausanne, Lausanne,
  Switzerland}
\affil[c]{INRIA, Rennes, France}
\let\oldthefootnote\thefootnote
\renewcommand{\thefootnote}{\fnsymbol{footnote}}
\footnotetext[1]{Principal Corresponding Author}
\renewcommand{\thefootnote}{}
\footnotetext{ {\em Email Addresses : } 
David Hammond (\url{hammond@uoregon.edu})
Pierre Vandergheynst (\url{pierre.vandergheynst@epfl.ch})
R\'emi Gribonval (\url{remi.gribonval@inria.fr})}

\let\thefootnote\oldthefootnote

\footnotetext[1]{This work performed while DKH was at EPFL}
\footnotetext[2]{This work was supported in part by the EU Framework 7
   FET-Open project FP7-ICT-225913-SMALL : Sparse Models, Algorithms
   and Learning for Large-Scale Data}

\maketitle

\begin{abstract}
 
  We propose a novel method for constructing wavelet transforms of
  functions defined on the vertices of an arbitrary finite weighted
  graph.  Our approach is based on defining scaling using the the
  graph analogue of the Fourier domain, namely the spectral
  decomposition of the discrete graph Laplacian $\L$.  Given a wavelet
  generating kernel $g$ and a scale parameter $t$, we define the
  scaled wavelet operator $T_g^t = g(t\L)$.  The spectral graph
  wavelets are then formed by localizing this operator by applying it
  to an indicator function.  Subject to an admissibility condition on
  $g$, this procedure defines an invertible transform.  We explore the
  localization properties of the wavelets in the limit of fine scales.
  Additionally, we present a fast Chebyshev polynomial approximation
  algorithm for computing the transform that avoids the need for
  diagonalizing $\L$. We highlight potential applications of the
  transform through examples of wavelets on graphs corresponding to a
  variety of different problem domains.

\end{abstract}

\section{Introduction}

Many interesting scientific problems involve analyzing and
manipulating structured data. Such data often consist of sampled
real-valued functions defined on domain sets themselves having some
structure. The simplest such examples can be described by scalar
functions on regular Euclidean spaces, such as time series data,
images or videos. However, many interesting applications involve data
defined on more topologically complicated domains. Examples include
data defined on network-like structures, data defined on manifolds or
irregularly shaped domains, and data consisting of ``point clouds'',
such as collections of feature vectors with associated labels.  As
many traditional methods for signal processing are designed for data
defined on regular Euclidean spaces, the development of methods that
are able to accommodate complicated data domains is an important
problem.

Many signal processing techniques are based on transform methods,
where the input data is represented in a new basis before analysis or
processing. One of the most successful types of transforms in use is
wavelet analysis. Wavelets have proved over the past 25 years to be an
exceptionally useful tool for signal processing. Much of the power of
wavelet methods comes from their ability to simultaneously localize
signal content in both space and frequency. For signals whose primary
information content lies in localized singularities, such as step
discontinuities in time series signals or edges in images, wavelets
can provide a much more compact representation than either the
original domain or a transform with global basis elements such as the
Fourier transform.
An enormous body of literature exists for describing and exploiting
this wavelet sparsity. We include a few representative references for
applications to signal compression
\cite{Shapiro1993,Said1996,Hilton1997,Buccigrossi1999,Taubman2002},
denoising
\cite{Donoho1994,Chang2000,Sendur2002,Portilla2003,Daubechies2005},
and inverse problems including deconvolution 
\cite{Starck1994,Donoho1995,Miller1995,Nowak2000,Bioucas-Dias2006}.
As the individual waveforms comprising the wavelet transform are self
similar, wavelets are also useful for constructing scale invariant
descriptions of signals. This property can be exploited for pattern
recognition problems where the signals to be recognized or classified
may occur at different levels of zoom \cite{Manthalkar2003}. In a
similar vein, wavelets can be used to characterize fractal
self-similar processes \cite{Flandrin1992}.

The demonstrated effectiveness of wavelet transforms for signal
processing problems on regular domains motivates the study of
extensions to irregular, non-euclidean spaces. In this paper, we
describe a flexible construction for defining wavelet transforms for
data defined on the vertices of a weighted graph. Our approach uses
only the connectivity information encoded in the edge weights, and
does not rely on any other attributes of the vertices (such as their
positions as embedded in 3d space, for example). As such, the
transform can be defined and calculated for any domain where the
underlying relations between data locations can be represented by a
weighted graph. This is important as weighted graphs provide an
extremely flexible model for approximating the data domains of a large
class of problems. 

Some data sets can naturally be modeled as scalar functions defined on
the vertices of graphs. For example, computer networks, transportation
(road, rail, airplane) networks or social networks can all be
described by weighted graphs, with the vertices corresponding to
individual computers, cities or people respectively.  The graph
wavelet transform could be useful for analyzing data defined on these
vertices, where the data is expected to be influenced by the
underlying topology of the network. As a mock example problem,
consider rates of infection of a particular disease among different
population centers. As the disease may be expected to spread by people
traveling between different areas, the graph wavelet transform based
on a weighted graph representing the transportation network may be
helpful for this type of data analysis.

Weighted graphs also provide a flexible generalization of regular grid
domains. By identifying the grid points with vertices and connecting
adjacent grid points with edges with weights inversely proportional to
the square of the distance between neighbors, a regular lattice can be
represented with weighted graph. A general weighted graph, however,
has no restriction on the regularity of vertices. For example points
on the original lattice may be removed, yielding a ``damaged grid'',
or placed at arbitrary locations corresponding to irregular
sampling. In both of these cases, a weighted graph can still be
constructed that represents the local connectivity of the underlying
data points. Wavelet transforms that rely upon regular spaced samples
will fail in these cases, however transforms based on weighted graphs
may still be defined.

Similarly, weighted graphs can be inferred from mesh descriptions for
geometrical domains. An enormous literature exists on techniques for
generating and manipulating meshes; such structures are widely used in
applications for computer graphics and numerical solution of partial
differential equations. The transform methods we will describe thus
allow the definition of a wavelet transform for data defined on any
geometrical shape that can be described by meshes.

Weighted graphs can also be used to describe the similarity
relationships between ``point clouds'' of vectors.  Many approaches
for machine learning or pattern recognition problems involve
associating each data instance with a collection of feature vectors
that hopefully encapsulate sufficient information about the data point
to solve the problem at hand. For example, for machine vision problems
dealing with object recognition, a common preprocessing step involves
extracting keypoints and calculating the Scale Invariant Feature
Transform (SIFT) features \cite{Lowe1999}.  In many automated systems
for classifying or retrieving text, word frequencies counts are used
as feature vectors for each document \cite{Apte1994}. After such
feature extraction, each data point may be associated to a feature
vector $v_m \in \R^N$, where $N$ may be very large depending on the
application. For many problems, the local distance relationships
between data points are crucial for successful learning or
classification. These relationships can be encoded in a weighted graph
by considering the data points as vertices and setting the edge
weights equal to a distance metric $A_{m,n}=d(v_m,v_n)$ for some
function $d:\R^N \times \R^N \to \R$. The spectral graph wavelets
applied to such graphs derived from point clouds could find a number
of uses, including regression problems involving learning or
regularizing a scalar function defined on the data points.

Classical wavelets are constructed by translating and scaling a single
``mother'' wavelet. The transform coefficients are then given by the
inner products of the input function with these translated and scaled
waveforms. Directly extending this construction to arbitrary weighted
graphs is problematic, as it is unclear how to define scaling and
translation on an irregular graph.  We approach this problem by
working in the spectral graph domain, i.e. the space of eigenfunctions
of the graph Laplacian $\L$.  This tool from spectral graph theory
\cite{Chung1997}, provides an analogue of the Fourier transform for
functions on weighted graphs.  In our construction, the wavelet
operator at unit scale is given as an operator valued function $T_g =
g(\L)$ for a generating kernel $g$.  Scaling is then defined in the
spectral domain, i.e. the operator $T^t_g$ at scale $t$ is given by
$g(t\L)$. Applying this operator to an input signal $f$ gives the
wavelet coefficients of $f$ at scale $t$. These coefficients are
equivalent to inner products of the signal $f$ with the individual
graph wavelets. These wavelets can be calculated by applying this
operator to a delta impulse at a single vertex, i.e. $\psi_{t,m} =
T^t_g \delta_m$.  We show that this construction is analogous to the
1-d wavelet transform for a symmetric wavelet, where the transform is
viewed as a Fourier multiplier operator at each wavelet scale.

In this paper we introduce this spectral graph wavelet transform and
study several of its properties. We show that in the fine scale limit,
for sufficiently regular $g$, the wavelets exhibit good localization
properties. With continuously defined spatial scales, the transform is
analogous to the continuous wavelet transform, and we show that it is
formally invertible subject to an admissibility condition on the
kernel $g$. Sampling the spatial scales at a finite number of values
yields a redundant, invertible transform with overcompleteness equal
to the number of spatial scales chosen. We show that in this case the
transform defines a frame, and give a condition for computing the
frame bounds depending on the selection of spatial scales.

While we define our transform in the spectral graph domain, directly
computing it via fully diagonalizing the Laplacian operator is
infeasible for problems with size exceeding a few thousand
vertices. We introduce a method for approximately computing the
forward transform through operations performed directly in the vertex
domain that avoids the need to diagonalize the Laplacian. By
approximating the kernel $g$ with a low dimensional Chebyshev
polynomial, we may compute an approximate forward transform in a
manner which accesses the Laplacian only through matrix-vector
multiplication. This approach is computationally efficient if the
Laplacian is sparse, as is the case for many practically relevant
graphs.
 
We show that computation of the pseudoinverse of the overcomplete
spectral graph wavelet transform is compatible with the Chebyshev
polynomial approximation scheme. Specifically, the pseudoinverse may
be calculated by an iterative conjugate gradient method that requires
only application of the forward transform and its adjoint, both of
which may be computed using the Chebyshev polynomial approximation
methods.

Our paper is structured as follows. Related work is discussed in
Section \ref{sec:related}. We review the classical wavelet transform
in Section \ref{sec:cwt}, and highlight the interpretation of the
wavelet acting as a Fourier multiplier operator. We then set our
notation for weighted graphs and introduce spectral graph theory in
Section \ref{sec:sgwt}. The spectral graph wavelet transform is
defined in Section \ref{sec:sgwt}. In Section \ref{sec:prop} we
discuss and prove some properties of the transform. Section
\ref{sec:polyapprox} is dedicated to the polynomial approximation and
fast computation of the transform. The inverse transform is discussed
in section \ref{sec:rec}.  Finally, several examples of the transform
on domains relevant for different problems are shown in Section
\ref{sec:examp}.

\subsection{Related Work}
\label{sec:related}

Since the original introduction of wavelet theory for square
integrable functions defined on the real line, numerous authors have
introduced extensions and related transforms for signals on the plane
and higher dimensional spaces. By taking separable products of one
dimensional wavelets, one can construct orthogonal families of
wavelets in any dimension \cite{Mallat1998}. However, this yields
wavelets with often undesirable bias for coordinate axis directions. A
large family of alternative multiscale transforms has been developed
and used extensively for image processing, including Laplacian
pyramids \cite{Burt1983}, steerable wavelets \cite{Simoncelli1992},
complex dual-tree wavelets \cite{Kingsbury2001}, curvelets
\cite{Candes2003}, and bandlets \cite{Peyre2008obb}.  Wavelet
transforms have also been defined for certain non-Euclidean manifolds,
most notably the sphere \cite{Antoine1999,Wiaux2008} and other conic
sections \cite{Antoine2008}.

Previous authors have explored wavelet transforms on graphs, albeit
via different approaches to those employed in this paper.  Crovella
and Kolaczyk \cite{Crovella2003} defined wavelets on unweighted graphs
for analyzing computer network traffic.  Their construction was based
on the n-hop distance, such that the value of a wavelet centered at a
vertex $n$ on vertex $m$ depended only on the shortest-path distance
between $m$ and $n$. The wavelet values were such that the sum over
each n-hop annulus equaled the integral over an interval of a given
zero-mean function, thus ensuring that the graph wavelets had zero
mean. Their results differ from ours in that their construction made
no use of graph weights and no study of the invertibility or frame
properties of the transform was done.  Smalter
et. al. \cite{Smalter2009} used the graph wavelets of Crovella and
Kolaczyk as part of a larger method for measuring structural
differences between graphs representing chemical structures, for
machine learning of chemical activities for virtual drug screening.

Maggioni and Coiffmann \cite{Coifman2006} introduced ``diffusion
wavelets'', a general theory for wavelet decompositions based on
compressed representations of powers of a diffusion operator. The
diffusion wavelets were described with a framework that can apply on
smooth manifolds as well as graphs.  Their construction interacts with
the underlying graph or manifold space through repeated applications
of a diffusion operator $T$, analogously to how our construction is
parametrized by the choice of the graph Laplacian $\L$.  The largest
difference between their work and ours is that the diffusion wavelets
are designed to be orthonormal. This is achieved by running a
localized orthogonalization procedure after applying dyadic powers of
$T$ at each scale to yield nested approximation spaces, wavelets are
then produced by locally orthogonalizing vectors spanning the
difference of these approximation spaces.  While an orthogonal
transform is desirable for many applications, notably operator and
signal compression, the use of the orthogonalization procedure
complicates the construction of the transform, and somewhat obscures
the relation between the diffusion operator $T$ and the resulting
wavelets. In contrast our approach is conceptually simpler, gives a
highly redundant transform, and affords finer control over the
selection of wavelet scales.

Geller and Mayeli \cite{Geller2009} studied a construction for
wavelets on compact differentiable manifolds that is formally similar
to our approach on weighted graphs. In particular, they define scaling
using a pseudodifferential operator $tL e^{-tL}$, where $L$ is the
manifold Laplace-Beltrami operator and $t$ is a scale parameter, and
obtain wavelets by applying this to a delta impulse. They also study
the localization of the resulting wavelets, however the methods and
theoretical results in their paper are different as they are in the
setting of smooth manifolds.

\section{Classical Wavelet Transform}
\label{sec:cwt}

We first give an overview of the classical Continuous Wavelet
Transform (CWT) for $L^2(\R)$, the set of square integrable real valued
functions. We will describe the forward transform and its formal
inverse, and then show how scaling may be expressed in the Fourier
domain. These expressions will provide an analogue that we will later
use to define the Spectral Graph Wavelet Transform.

In general, the CWT will be generated by the choice of a single
``mother'' wavelet $\psi$. Wavelets at different locations and spatial
scales are formed by translating and scaling the mother wavelet. We
write this by
\begin{equation}
\psi_{s,a}(x) = \frac{1}{s} \psi\left(\frac{x-a}{s}\right)
\end{equation}
This scaling convention preserves the $L^1$ norm of the
wavelets. Other scaling conventions are common, especially those
preserving the $L^2$ norm, however in our case the $L^1$ convention
will be more convenient. We restrict ourselves to positive scales
$s>0$.

For a given signal $f$, the wavelet coefficient at scale $s$ and
location $a$ is given by the inner product of $f$ with the wavelet
$\psi_{s,a}$, i.e.
\begin{equation}
W_f(s,a)= \int_{-\infty}^{\infty} \frac{1}{s}\psi^*\left(\frac{x-a}{s}\right) f(x) dx
\end{equation}

The CWT may be inverted provided that the wavelet $\psi$
satisfies the  admissibility condition 
\begin{equation}\label{eq:admissibility}
\int_0^{\infty} \frac{|\hat{\psi}(\omega)|^2}{\omega}d\omega = C_{\psi} < \infty
\end{equation}
This condition implies, for continuously differentiable $\psi$, that
$\hat{\psi}(0) = \int \psi(x)dx =0$, so $\psi$ must be zero mean.

Inversion of the CWT is given by the following relation
\cite{Grossmann1984}
\begin{equation}\label{eq:cwt_inverse}
f(x)=\frac{1}{C_\psi} \int_0^{\infty} \int_{-\infty}^{\infty}
W_f(s,a)\psi_{s,a}(x) \frac{da ds}{s}
\end{equation}





This method of constructing the wavelet transform proceeds by
producing the wavelets directly in the signal domain, through scaling
and translation. However, applying this construction directly to
graphs is problematic. For a given function $\psi(x)$ defined on the
vertices of a weighted graph, it is not obvious how to define
$\psi(sx)$, as if $x$ is a vertex of the graph there is no
interpretation of $sx$ for a real scalar $s$. Our approach to this
obstacle is to appeal to the Fourier domain. We will first show that
for the classical wavelet transform, scaling can be defined in the
Fourier domain. The resulting expression will give us a basis to
define an analogous transform on graphs.

For the moment, we consider the case where the scale parameter is
discretized while the translation parameter is left continuous. While
this type of transform is not widely used, it will provide us with the
closest analogy to the spectral graph wavelet transform.  For a fixed
scale $s$, the wavelet transform may be interpreted as an operator
taking the function $f$ and returning the function $T^{s} f (a) =
W_f(s,a)$. In other words, we consider the translation parameter as
the independent variable of the function returned by the operator
$T^{s}$.
Setting 
\begin{equation}\label{def:psibar}
\bar{\psi}_s(x) = 
\frac{1}{s}\psi^*\left(\frac{-x}{s}\right),
\end{equation}
we see that this operator is given by convolution, i.e.
\begin{align}
(T^{s} f)(a) 
&=\int_{-\infty}^{\infty} \frac{1}{s}\psi^*\left(\frac{x-a}{s}\right) f(x) dx \\ \notag
&=\int_{-\infty}^{\infty} \bar{\psi}_{s}(a-x)f(x)dx \\ \notag
&= (\bar{\psi}_{s} \star f)(a)
\end{align}
Taking the Fourier transform and applying the convolution theorem yields
\begin{equation}
\widehat{T^{s}f}(\omega) = \hat{\bar{\psi}}_{s}(\omega) \hat{f}(\omega)
\end{equation}
Using the scaling properties of the Fourier transform and the
definition (\ref{def:psibar}) gives
\begin{equation}
\hat{\bar{\psi}}_{s}(\omega) = \hat{\psi^*}(s \omega)
\end{equation}
Combining these and inverting the transform we may write
\begin{equation}\label{eq:tsf}
(T^s f)(x) = \frac{1}{2\pi}\int_{-\infty}^{\infty} e^{i\omega x} \hat{\psi}^*(s\omega) \hat{f}(\omega) d\omega
\end{equation}

In the above expression, the scaling $s$ appears only in the argument
of $\hat{\psi}^*(s\omega)$, showing that the scaling operation can be
completely transferred to the Fourier domain. The above expression
makes it clear that the wavelet transform at each scale $s$ can be
viewed as a Fourier multiplier operator, determined by filters that
are derived from scaling a single filter $\hat{\psi}^*(\omega)$. This
can be understood as a bandpass filter, as $\hat{\psi}(0)=0$ for
admissible wavelets.  Expression (\ref{eq:tsf}) is the analogue that
we will use to later define the Spectral Graph Wavelet Transform.

Translation of the wavelets may be defined through ``localizing'' the
wavelet operator by applying it to an impulse. 
Writing $\delta_a(x) = \delta(x-a)$, one has
\begin{equation}
(T^s \delta_a)(x) = \frac{1}{s}\psi^*\left(\frac{a-x}{s}\right)
\end{equation}
For real valued and even wavelets this reduces to $(T^s\delta_a)(x) =
\psi_{a,s}(x)$.



\section{Weighted Graphs and Spectral Graph Theory}
\label{sec:sgt}

The previous section showed that the classical wavelet transform could
be defined without the need to express scaling in the original signal
domain. This relied on expressing the wavelet operator in the Fourier
domain. Our approach to defining wavelets on graphs relies on
generalizing this to graphs; doing so requires the analogue of the
Fourier transform for signals defined on the vertices of a weighted
graph. This tool is provided by Spectral Graph Theory. In this section
we fix our notation for weighted graphs, and motivate and define the
Graph Fourier transform.

\subsection{Notation for Weighted Graphs}

A weighted graph $G = \{E,V,w\}$ consists of a set of vertices $V$, a
set of edges $E$, and a weight function $w:E\to\R^+$ which assigns a
positive weight to each edge. We consider here only finite graphs
where $|V| = N <\infty$.  The adjacency matrix $A$ for a weighted
graph $G$ is the $N\times N$ matrix with entries $a_{m,n}$ where
\begin{equation}
a_{m,n} = 
\begin{cases}
w(e) \mbox{ if $e\in E$ connects vertices $m$ and $n$} \\
0 \mbox{ otherwise}
\end{cases}
\end{equation}
In the present work we consider only undirected graphs, which
correspond to symmetric adjacency matrices. We do not consider the
possibility of negative weights.

A graph is said to have loops if it contain edges that connect a
single vertex to itself. Loops imply the presence of nonzero diagonal
entries in the adjacency matrix. As the existence of loops presents no
significant problems for the theory we describe in this paper, we do
not specifically disallow them.

For a weighted graph, the degree of each vertex $m$, written as
$d(m)$, is defined as the sum of the weights of all the edges incident
to it. This implies $d(m) = \sum_n a_{m,n}$. We define the matrix $D$
to have diagonal elements equal to the degrees, and zeros elsewhere.

Every real valued function $f:V\to\R$ on the vertices of the graph $G$
can be viewed as a vector in $\R^N$, where the value of $f$ on each
vertex defines each coordinate. This implies an implicit numbering of
the vertices. We adopt this identification, and will write $f\in R^N$
for functions on the vertices of the graph, and $f(m)$ for the value
on the $m^{th}$ vertex.

Of key importance for our theory is the graph Laplacian operator
$\L$. The non-normalized Laplacian is defined as $\L=D-A$. It can be
verified that for any $f\in \R^N$, $\L$ satisfies
\begin{equation}
(\L f)(m) = \sum_{m\sim n} w_{m,n} \cdot ( f(m) - f(n))
\end{equation}
where the sum over $m\sim n$ indicates summation over all vertices $n$
that are connected to the vertex $m$, and $w_{m,n}$ denotes the weight
of the edge connecting $m$ and $n$.

For a graph arising from a regular mesh, the graph Laplacian
corresponds to the standard stencil approximation of the continuous
Laplacian (with a difference in sign). Consider the graph defined by
taking vertices $v_{m,n}$ as points on a regular two dimensional grid,
with each point connected to its four neighbors with weight
$1/{(\delta x)^2}$, where $\delta x$ is the distance between adjacent
grid points. Abusing the index notation, for a function $f=f_{m,n}$
defined on the vertices, applying the graph Laplacian to $f$ yields
\begin{equation} \label{eq:laplacian_stencil}
(\L f)_{m,n}=(4 f_{m,n}-f_{m+1,n}-f_{m-1,n} - f_{m,n+1} - f_{m,n-1})/(\delta x)^2
\end{equation}
which is the standard 5-point stencil for approximating $-\nabla^2 f$.

Some authors define and use an alternative, normalized form of the
Laplacian, defined as 
\begin{equation}
\L^{norm} = D^{-1/2} \L D^{-1/2} = I-D^{-1/2} A D^{-1/2}
\end{equation}
It should be noted that $\L$ and $\L^{norm}$ are not similar matrices,
in particular their eigenvectors are different.  As we shall see in
detail later, both operators may be used to define spectral graph
wavelet transforms, however the resulting transforms will not be
equivalent. Unless noted otherwise we will use the non-normalized form
of the Laplacian, however much of the theory presented in this paper
is identical for either choice. We consider that the selection of the
appropriate Laplacian for a particular problem should depend on the
application at hand.

For completeness, we note the following. The graph Laplacian can be
defined for graphs arising from sampling points on a differentiable
manifold. The regular mesh example described previously is a simple
example of such a sampling process.  With increasing sampling density,
by choosing the weights appropriately the normalized graph Laplacian
operator will converge to the intrinsic Laplace-Beltrami operator
defined for differentiable real valued functions on the
manifold. Several authors have studied this limiting process in
detail, notably \cite{Hein2005,Singer2006,Belkin2008}.

\subsection{Graph Fourier Transform}

\newcommand{\evec}{\chi}

On the real line, the complex exponentials $e^{i\omega x}$ defining
the Fourier transform are eigenfunctions of the one-dimensional
Laplacian operator $\frac{d}{dx^2}$. The inverse Fourier transform
\begin{equation}
f(x) = \frac{1}{2\pi} \int \hat{f}(\omega) e^{i\omega x}d\omega
\end{equation}
can thus be seen as the expansion of $f$ in terms of the
eigenfunctions of the Laplacian operator. 

The graph Fourier transform is defined in precise analogy to the
previous statement. As the graph Laplacian $\L$ is a real symmetric
matrix, it has a complete set of orthonormal eigenvectors. We denote
these by $\evec_{\l}$ for $\l=0, \hdots ,N-1$, with associated eigenvalues
$\lambda_{\l}$
\begin{equation}
\L \evec_{\l} = \lambda_{\l} \evec_{\l}
\end{equation}
As $\L$ is symmetric, each of the $\lambda_{\l}$ are real. For the graph
Laplacian, it can be shown that the eigenvalues are all non-negative,
and that $0$ appears as an eigenvalue with multiplicity equal to the
number of connected components of the graph
\cite{Chung1997}. 
%
%
Henceforth, we assume the graph $G$ to be
connected, we may thus order the eigenvalues such that
\begin{equation}
0=\lambda_0 < \lambda_1 \leq \lambda_2 ... \leq \lambda_{N-1}
\end{equation}

For any function $f\in\R^{N}$ defined on the vertices of $G$, its
graph Fourier transform $\hat{f}$ is defined by
\begin{equation}
\hat{f}(\l) = \ip{\evec_{\l}}{f} = \sum_{n=1}^N \evec^*_{\l}(n) f(n)
\end{equation}
The inverse transform reads as
\begin{equation}
f(n) = \sum_{\l=0}^{N-1} \hat{f}(\l) \evec_{\l}(n)
\end{equation}

The Parseval relation holds for the graph Fourier transform, in
particular for any $f,h \in \R^N$
\begin{equation} \label{eq:parseval}
\ip{f}{h} = \ip{\hat{f}}{\hat{g}}.
\end{equation}

\section{Spectral Graph Wavelet Transform}
\label{sec:sgwt}

Having defined the analogue of the Fourier transform for functions
defined on the vertices of weighted graphs, we are now ready to define
the spectral graph wavelet transform (SGWT). The transform will be determined
by the choice of a kernel function $g : \R^+ \to \R^+$, which is
analogous to Fourier domain wavelet $\hat{\psi}^*$ in equation
\ref{eq:tsf}. This kernel $g$ should behave as a band-pass filter,
i.e. it satisfies $g(0)=0$ and $\lim_{x\to\infty} g(x)=0$.  We will
defer the exact specification of the kernel $g$ that we use until
later.

\subsection{Wavelets}

The spectral graph wavelet transform is generated by wavelet operators
that are operator-valued functions of the Laplacian. One may define a
measurable function of a bounded self-adjoint linear operator on a
Hilbert space using the continuous functional calculus
\cite{Reed1980}. This is achieved using the Spectral representation of
the operator, which in our setting is equivalent to the graph Fourier
transform defined in the previous section. In particular, for our
spectral graph wavelet kernel $g$, the wavelet operator $T_g = g(\L)$
acts on a given function $f$ by modulating each Fourier mode as
\begin{equation}
\widehat{T_g f}(\l) = g(\lambda_\l) \hat{f}(\l)
\end{equation}
Employing the inverse Fourier transform yields
\begin{equation}\label{eq:t_g_operator_fourier_exp}
(T_g f)(m) = \sum_{\l=0}^{N-1} g(\lambda_\l) \hat{f}(\l) \evec_\l(m)
\end{equation}

The wavelet operators at scale $t$ is then defined by $T_g^t =
g(t\L)$. It should be emphasized that even though the ``spatial
domain'' for the graph is discrete, the domain of the kernel $g$ is
continuous and thus the scaling may be defined for any positive real number
$t$. 

The spectral graph wavelets are then realized through localizing these
operators by applying them to the impulse on a single vertex, i.e.
\begin{equation}
\psi_{t,n} = T^t_g \delta_n
\end{equation}
Expanding this explicitly in the graph domain shows
\begin{equation} \label{eq:psi_fourier_exp}
\psi_{t,n}(m) = \sum_{\l=0}^{N-1} g(t\lambda_{\l}) \evec_\l^*(n)\evec_\l(m)
\end{equation}

Formally, the wavelet coefficients of a given function $f$ are
produced by taking the inner product with these wavelets, as
\begin{equation}
W_f(t,n) = \ip{\psi_{t,n}}{f}
\end{equation}
Using the orthonormality of the $\{\evec_\l\}$, it can be seen that the
wavelet coefficients can also be achieved directly from the wavelet
operators, as
\begin{equation} \label{eq:w_fourier_exp}
W_f(t,n) = \left(T_g^tf \right) (n) = \sum_{\l=0}^{N-1}
g(t\lambda_\l)\hat{f}(\l)\evec_\l(n)
\end{equation}

\subsection{Scaling functions}
\newcommand{\scalf}{\phi} 

By construction, the spectral graph wavelets $\psi_{t,n}$ are all
orthogonal to the null eigenvector $\evec_0$, and nearly orthogonal to
$\evec_\l$ for $\lambda_\l$ near zero. In order to stably represent the
low frequency content of $f$ defined on the vertices of the graph, it
is convenient to introduce a second class of waveforms, analogous to
the lowpass residual scaling functions from classical wavelet
analysis. These spectral graph scaling functions have an analogous
construction to the spectral graph wavelets. They will be determined
by a single real valued function $h:\R^+\to\R$, which acts as a
lowpass filter, and satisfies $h(0)>0$ and $h(x)\to 0$ as $x\to
\infty$. The scaling functions are then given by $\scalf_n = T_h
\delta_n = h(\L) \delta_n$, and the coefficients by $S_f(n) =
\ip{\scalf_n}{f}$.

\begin{figure} 
  \centering{
    \includegraphics[width=.5\columnwidth]{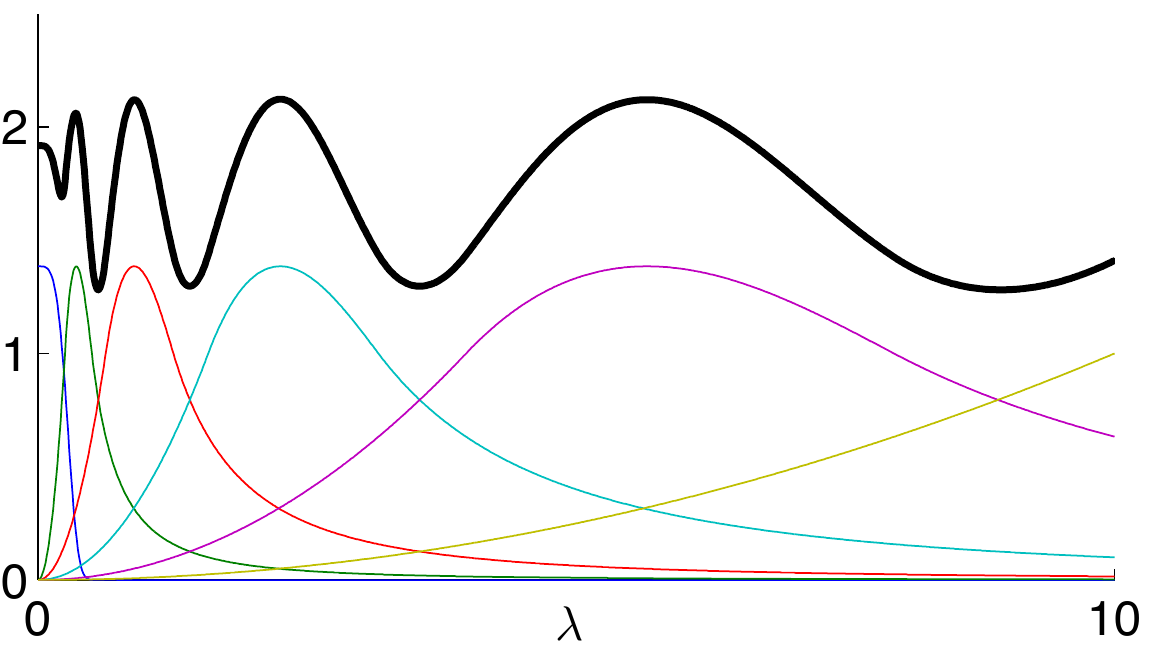}
  }
  \caption{ Scaling function $h(\lambda)$ (blue curve),  wavelet generating
    kernels $g(t_j\lambda)$, and sum of squares $G$ (black curve), for $J=5$ scales,
    $\lambda_{max}=10$, $K=20$.  Details in Section
    \ref{sec:sgwt_design}.}\label{fig:spectral_design}
\end{figure}

Introducing the scaling functions helps ensure stable recovery of the
original signal $f$ from the wavelet coefficients when the scale
parameter $t$ is sampled at a discrete number of values $t_j$.  As we
shall see in detail in Section \ref{sec:wavelet_frames}, stable
recovery will be assured if the quantity $G(\lambda)=h(\lambda)^2 +
\sum_{j=1}^J g(t_j \lambda)^2$ is bounded away from zero on the
spectrum of $\L$.  Representative choices for $h$ and $g$ are shown in
figure \ref{fig:spectral_design}; the exact specification of $h$ and
$g$ is deferred to Section \ref{sec:sgwt_design}.


Note that the scaling functions defined in this way are present
merely to smoothly represent the low frequency content on the
graph. They do not generate the wavelets $\psi$ through the two-scale
relation as for traditional orthogonal wavelets. The design of the
scaling function generator $h$ is thus uncoupled from the choice of
wavelet kernel $g$, provided reasonable tiling for $G$ is achieved.

\section{Transform properties}
\label{sec:prop}
In this section we detail several properties of the spectral graph
wavelet transform. We first show an inverse formula for the transform
analogous to that for the continuous wavelet transform. We examine the
small-scale and large-scale limits, and show that the wavelets are
localized in the limit of small scales. Finally we discuss
discretization of the scale parameter and the resulting wavelet
frames.

\subsection{Continuous SGWT  Inverse}

In order for a particular transform to be useful for signal
processing, and not simply signal analysis, it must be possible to
reconstruct from a given set of transform coefficients. We will show
that the spectral graph wavelet transform admits an inverse formula
analogous to (\ref{eq:cwt_inverse}) for the continuous wavelet
transform. 

Intuitively, the wavelet coefficient $W_f(t,n)$ provides a measure of
``how much of'' the wavelet $\psi_{t,n}$ is present in the signal
$f$. This suggests that the original signal may be recovered by
summing the wavelets $\psi_{t,n}$ multiplied by each wavelet
coefficient $W_f(t,n)$. The reconstruction formula below shows that
this is indeed the case, subject to a non-constant weight $dt/t$.

\begin{lemma} \label{lemma:continuous_inverse}
If the SGWT kernel $g$ satisfies the admissibility condition
\begin{equation}
\int_0^{\infty} \frac{g^2(x)} {x} dx = C_g < \infty ,
\end{equation}
and $g(0)=0$, then 
\begin{equation} \label{eq:lemma1conclusion}
\frac{1}{C_g}\sum_{n=1}^N \int_0^{\infty} W_f(t,n) \psi_{t,n}(m) \frac{dt}{t}
=f^\#(m)
\end{equation}
where $f^\# = f-\ip{\evec_0}{f}\evec_0$. In particular, the
complete reconstruction is then given by $f = f^\# +
\hat{f}(0)\evec_0$.

\end{lemma}

\begin{proof}
  Using (\ref{eq:psi_fourier_exp}) and (\ref{eq:w_fourier_exp}) to
  express $\psi_{t,n}$ and $W_f(t,n)$ in the graph Fourier basis, the
  l.h.s. of the above becomes
\begin{align}
\frac{1}{C_g} & \int_0^{\infty} \frac{1}{t} \sum_n \left(
\sum_{\l} g(t\lambda_{\l}) \evec_{\l}(n) \hat{f}(\l)
\sum_{\l'} g(t\lambda_{\l'}) \evec^*_{\l'}(n)\evec_{\l'}(m) \right) dt \\
= \frac{1}{C_g} & \int_0^{\infty} \frac{1}{t} \left(
\sum_{\l,\l'} g(t\lambda_{\l'}) g(t\lambda_{\l}) \hat{f}({\l}) \evec_{\l'}(m) 
\sum_n \evec^*_{\l'}(n) \evec_{\l}(n) \right) dt
\end{align}
The orthonormality of the $\evec_\l$ implies $\sum_n \evec_{\l'}^*(n)\evec_{\l}(n) =
\delta_{\l,\l'}$, inserting this above and summing over ${\l'}$ gives
\begin{equation} \label{eq:lemma1step2}
=\frac{1}{C_g} \sum_{\l} \left( 
\int_0^{\infty} \frac{g^2(t\lambda_{\l})}{t} dt \right) 
\hat{f}(\l) \evec_{\l}(m)\\
\end{equation}
If $g$ satisfies the admissibility condition, then the substitution
$u=t\lambda_{\l}$ shows that $\int \frac{g^2(t\lambda_{\l})}{t}dt = C_g$
independent of $\l$, except for when $\lambda_\l=0$ at $\l=0$ when the
integral is zero. The expression (\ref{eq:lemma1step2}) can be seen as
the inverse Fourier transform evaluated at vertex $m$, where the $\l=0$
term is omitted. This omitted term is exactly equal to
$\ip{\evec_0}{f}\evec_0 = \hat{f}(0) \evec_0$, which proves the desired
result.
\end{proof}

Note that for the non-normalized Laplacian, $\evec_0$ is constant on
every vertex and $f^\#$ above corresponds to removing the mean of
$f$.  Formula (\ref{eq:lemma1conclusion}) shows that the mean of $f$
may not be recovered from the zero-mean wavelets.  The situation is
different from the analogous reconstruction formula
(\ref{eq:cwt_inverse}) for the CWT, which shows the somewhat
counterintuitive result that it is possible to recover a non zero-mean
function by summing zero-mean wavelets. This is possible on the real
line as the Fourier frequencies are continuous; the fact that it is not
possible for the SGWT should be considered a consequence of the
discrete nature of the graph domain.

While it is of theoretical interest, we note that this continuous
scale reconstruction formula may not provide a practical
reconstruction in the case when the wavelet coefficients may only be
computed at a discrete number of scales, as is the case for finite
computation on a digital computer. We shall revisit this and discuss
other reconstruction methods in sections \ref{sec:wavelet_frames} and
\ref{sec:rec}. 

\subsection{Localization in small scale limit}

One of the primary motivations for the use of wavelets is that they
provide simultaneous localization in both frequency and time (or
space).  It is clear by construction that if the kernel $g$ is
localized in the spectral domain, as is loosely implied by our use of
the term bandpass filter to describe it, then the associated spectral
graph wavelets will all be localized in frequency. In order to be able
to claim that the spectral graph wavelets can yield localization in
both frequency and space, however, we must analyze their behaviour in
the space domain more carefully.

For the classical wavelets on the real line, the space localization is
readily apparent : if the mother wavelet $\psi(x)$ is well localized
in the interval $[-\epsilon,\epsilon]$, then the wavelet
$\psi_{t,a}(x)$ will be well localized within $[a-\epsilon
t,a+\epsilon t]$.  In particular, in the limit as $t\to 0$,
$\psi_{t,a}(x) \to 0$ for $x\neq a$.  The situation for the spectral
graph wavelets is less straightforward to analyze because the scaling
is defined implicitly in the Fourier domain. We will nonetheless show
that, for $g$ sufficiently regular near 0, the normalized spectral
graph wavelet $\psi_{t,j}/\norm{\psi_{t,j}}$ will vanish on vertices
sufficiently far from $j$ in the limit of fine scales, i.e. as $t \to
0$. This result will provide a quantitative statement of the
localization properties of the spectral graph wavelets.

One simple notion of localization for $\psi_{t,n}$ is given by its
value on a distant vertex $m$, e.g. we should expect $\psi_{t,n}(m)$
to be small if $n$ and $m$ are separated, and $t$ is small. Note that
$\psi_{t,n}(m) = \ip{\psi_{t,n}}{\delta_m} =
\ip{T^t_g\delta_n}{\delta_m}$. The operator $T^t_g =g(t\L)$ is
self-adjoint as $\L$ is self adjoint.  This shows that $\psi_{t,n}(m)
= \ip{\delta_n}{T^t_g\delta_m}$, i.e. a matrix element of the operator
$T^t_g$.

Our approach is based on approximating $g(t\L)$ by a low order
polynomial in $\L$ as $t\to 0$. As is readily apparent by inspecting
equation (\ref{eq:t_g_operator_fourier_exp}), the operator $T^t_g$
depends only on the values of $g_t(\lambda)$ restricted to the
spectrum $\{\lambda_\l\}_{\l=0}^{N-1}$ of $\L$. In particular, it is
insensitive to the values of $g_t(\lambda)$ for
$\lambda>\lambda_{N-1}$. If $g(\lambda)$ is smooth in a neighborhood
of the origin, then as $t$ approaches $0$ the zoomed in $g_t(\lambda)$
can be approximated over the entire interval $[0,\lambda_{N-1}]$ by
the Taylor polynomial of $g$ at the origin. In order to transfer the
study of the localization property from $g$ to an approximating
polynomial, we will need to examine the stability of the wavelets
under perturbations of the generating kernel. This, together with the
Taylor approximation will allow us to examine the localization
properties for integer powers of the Laplacian $\L$.

In order to formulate the desired localization result, we must specify
a notion of distance between points $m$ and $n$ on a weighted
graph. We will use the shortest-path distance, i.e. the minimum number
of edges for any paths connecting $m$ and $n$ :
\begin{align} 
d_G(m,n) &= \argmin_s \{k_1, k_2, ... , k_s\} \\
& \mbox{ s.t. $m=k_1$, $n=k_s$, and $w_{k_r,k_{r+1}}>0$ for $1\leq r<s$. }
\end{align}
Note that as we have defined it, $d_G$ disregards the values of the
edge weights. In particular it defines the same distance function on
$G$ as on the binarized graph where all of the nonzero edge weights
are set to unit weight. 

We need the following elegant elementary result from graph theory \cite{Bondy2008}.
\begin{lemma}\label{lemma:pathcount}
  Let $G$ be a weighted graph, with adjacency matrix $A$. Let
  $B$ equal the adjacency matrix of the binarized graph, i.e.
  $B_{m,n}=0$ if $A_{m,n}=0$, and $B_{m,n}=1$  if  $A_{m,n}>0$. Let
  $\tilde{B}$ be the adjacency matrix with unit loops added on every
  vertex, e.g. $\tilde{B}_{m,n}=B_{m,n}$ for $m\neq n$ and
  $\tilde{B}_{m,n}=1$ for $m=n$.

  Then for each $s>0$, $(B^s)_{m,n}$ equals the number of paths of
  length $s$ connecting $m$ and $n$, and $(\tilde{B}^s)_{m,n}$ equals
  the number of all paths of length $r\leq s$ connecting $m$ and $n$.
\end{lemma}

We wish to use this to show that matrix elements of low powers of the
graph Laplacian corresponding to sufficiently separated vertices must
be zero. We first need the following

\begin{lemma}\label{lemma:posmat_power}
  Let $A=(a_{m,n})$ be an $M\times M$ matrix, and $B=(b_{m,n})$ an
  $M\times M$ matrix with non-negative entries, such that $B_{m,n}=0$
  implies $A_{m,n}=0$. Then, for all $s>0$, and all $m,n$,
  $(B^s)_{m,n} = 0$ implies that $(A^s)_{m,n} = 0$
  \begin{proof}
    By repeatedly writing matrix multiplication as explicit sums, one
    has
    \begin{align}
      (A^s)_{m,n} &= \sum a_{m,k_1} a_{k_1,k_2} ... a_{k_{s-1},n} \\
      (B^s)_{m,n} &= \sum b_{m,k_1} b_{k_1,k_2} ... b_{k_{s-1},n} 
    \end{align}
    where the sum is taken over all $s-1$ length sequences $k_1, k_2
    ... k_{s-1}$ with $1\leq k_r\leq M$. Fix $s$, $m$ and $n$. If
    $(B^s)_{m,n}=0$, then as every element of $B$ is non-negative,
    every term in the above sum must be zero. This implies that for
    each term, at least one of the $b_{q,r}$ factors must be
    zero. This implies by the hypothesis that the corresponding
    $a_{q,r}$ factor in the corresponding term in the sum for
    $(A^s)_{m,n}$ must be zero. As this holds for every term in the
    above sums, we have $(A^s)_{m,n}=0$
    \end{proof}
\end{lemma}

We now state the localization result for integer powers of the
Laplacian.
\begin{lemma}\label{lemma:laplacian_power}
  Let $G$ be a weighted graph, $\L$ the graph Laplacian (normalized or
  non-normalized) and $s>0$ an integer. For any two vertices $m$ and
  $n$, if $d_G(m,n)>s$ then $(\L^s)_{m,n} =0$.
  \begin{proof}
    Set the $\{0,1\}$ valued matrix $B$ such that $B_{q,r}=0$ if
    $\L_{q,r}=0$, otherwise $B_{q,r}=1$. The $B$ such defined will be
    the same whether $\L$ is normalized or non-normalized. $B$ is
    equal to the adjacency matrix of the binarized graph, but with
    $1$'s on the diagonal. According to Lemma \ref{lemma:pathcount},
    $(B^s)_{m,n}$ equals the number of paths of length $s$ or less
    from $m$ to $n$. As $d_G(m,n)>s$ there are no such paths, so
    $(B^s)_{m,n}=0$. But then Lemma \ref{lemma:posmat_power} shows
    that $(\L^s)_{m,n}=0$.
  \end{proof}
\end{lemma}

We now proceed to examining how perturbations in the kernel $g$ affect
the wavelets in the vertex domain. If two kernels $g$ and $g'$ are
close to each other in some sense, then the resulting wavelets should
be close to each other. More precisely, we have

\begin{lemma}\label{lemma:kernel_perturbation}
  Let $\psi_{t,n}=T^t_g \delta_n$ and $\psi'_{t,n}= T^t_{g'}$ be the
  wavelets at scale $t$ generated by the kernels $g$ and $g'$. If
  $\left|g(t\lambda)-g'(t\lambda)\right| \leq M(t)$ for all 
  $\lambda \in [0,\lambda_{N-1}]$, then
  $|\psi_{t,n}(m) - \psi'_{t,n}(m)| \leq M(t)$ for each vertex $m$.
  Additionally, 
  $\norm{\psi_{t,n}-\psi'_{t,n}}_2\leq \sqrt{N} M(t)$. 
  \begin{proof}
    First recall that
    $\psi_{t,n}(m)=\ip{\delta_m}{g(t\L)\delta_n}$. Thus,
 
    \begin{align} \label{eq:lemma_kernel_perturbation1}
      |\psi_{t,n}(m)-\psi'_{t,n}(m)| & =
      |\ip{\delta_m}{\left(g(t\L) - g'(t\L)\right) \delta_n}| \nt \\
      &=\left| \sum_\l \evec_\l(m) (g(t\lambda_\l)-g'(t\lambda_\l))
        \evec^*_\l(n)
      \right| \nt\\
      &\leq M(t) \sum_\l |\evec_\l(m) \evec_\l(n)^*| 
    \end{align}
    where we have used the Parseval relation (\ref{eq:parseval}) on the second
    line.  By Cauchy-Schwartz, the above sum over $\l$ is bounded by 1
    as
    \begin{equation}
      \sum_\l |\evec_\l(m)\evec^*_\l(n)| \leq 
      \left(\sum_\l | \evec_\l(m) |^2\right)^{1/2}
      \left(\sum_\l | \evec^*_\l(n) |^2\right)^{1/2}, 
    \end{equation}
    and $\sum_\l |\evec_\l(m)|^2=1$ for all $m$, as the $\evec_\l$
    form a complete orthonormal basis. Using this bound in
    (\ref{eq:lemma_kernel_perturbation1}) proves the first statement.

    The second statement follows immediately as
    \begin{align}
      \norm{\psi_{t,n}-\psi'_{t,n}}^2_2 = 
      \sum_m \left( \psi_{t,n}(m)-\psi'_{t,n}(m)\right)^2
      \leq \sum_m M(t)^2 = N M(t)^2
    \end{align}
\end{proof}
\end{lemma}

We will prove the final localization result for kernels $g$ which have
a zero of integer multiplicity at the origin. Such kernels can be
approximated by a single monomial for small scales.
\begin{lemma} \label{lemma:monomial_approx}
  Let $g$ be $K+1$ times continuously differentiable, satisfying
  $g(0)=0$, $g^{(r)}(0)=0$ for all $r<K$, and $g^{(K)}(0)=C\neq 0$. Assume
  that there is some $t'>0$ such that $|g^{(K+1)}(\lambda)| \leq B$ for
  all $\lambda\in [0,t'\lambda_{N-1}]$. Then, for 
  $g'(t\lambda)=(C/K!)(t\lambda)^K$ we have
\begin{equation}
  M(t) = \sup_{\lambda\in[0,\lambda_{N-1}]} |g(t\lambda)-g'(t\lambda)| 
\leq t^{K+1} \frac{\lambda_{N-1}^{K+1}}{(K+1)!}B
\end{equation}
for all $t<t'$.
\begin{proof}
  As the first $K-1$ derivatives of $g$ are zero, Taylor's formula
  with remainder shows, for any values of $t$ and $\lambda$,
\begin{equation}
  g(t\lambda) = C \frac{(t\lambda)^K}{K!} +
  g^{(K+1)}(x^*)\frac{(t\lambda)^{K+1}}{(K+1)!}
\end{equation}
for some $x^*\in[0,t\lambda]$. Now fix $t<t'$. For any
$\lambda\in[0,\lambda_{N-1}]$, we have $t\lambda < t'\lambda_{N-1}$,
and so the corresponding $x^*\in[0,t'\lambda_{N-1}]$, and so
$|g^{(K+1)}(x^*)\leq B$. This implies
\begin{equation}
  \left|g(t\lambda)-g'(t\lambda)\right|
\leq B \frac{t^{K+1} \lambda^{K+1}}{(K+1)!} 
\leq B \frac{t^{K+1} \lambda_{N-1}^{K+1}}{(K+1)!} 
\end{equation}
As this holds for all $\lambda\in[0,\lambda_{N-1}]$, taking the
sup over $\lambda$ gives the desired result.
\end{proof}
\end{lemma}

We are now ready to state the complete localization result. Note that
due to the normalization chosen for the wavelets, in general
$\psi_{t,n}(m)\to 0$ as $t \to 0$ for all $m$ and $n$. Thus a non
vacuous statement of localization must include a renormalization
factor in the limit of small scales.

\begin{theorem}\label{theorem:localization}
  Let $G$ be a weighted graph with Laplacian $\L$. Let $g$ be a kernel
  satisfying the hypothesis of Lemma \ref{lemma:monomial_approx}, with
  constants $t'$ and $B$.  Let $m$ and $n$ be vertices of $G$ such
  that $d_G(m,n) > K$. Then there exist constants $D$ and $t''$, such
  that 
  \begin{equation}
    \frac{\psi_{t,n}(m)}{\norm{\psi_{t,n}}} \leq D t
  \end{equation}
  for all $t<\min(t',t'')$.
\begin{proof}
  Set $g'(\lambda) = \frac{g^{(K)}(0)}{K!} \lambda^K$ and
  $\psi'_{t,n} = T^t_{g'} \delta_n$. We have
  \begin{equation}
    \psi'_{t,n}(m)=\frac{g^{(K)}(0)}{K!} t^K \ip{\delta_m}{\L^K\delta_n}=0
  \end{equation}
  by Lemma \ref{lemma:laplacian_power}, as $d_G(m,n)>K$. By the
  results of Lemmas \ref{lemma:kernel_perturbation} and
  \ref{lemma:monomial_approx}, we have
  \begin{equation}\label{eq:lemma_laplacian_power3}
  |\psi_{t,n}(m) -\psi'_{t,n}(m)| = |\psi_{t,n}(m)| \leq t^{K+1} C'
  \end{equation}
  for $C'=\frac{\lambda_{N-1}^{K+1}}{(K+1)!}B$. Writing $\psi_{t,n} =
  \psi'_{t,n} + (\psi_{t,n} - \psi'_{t,n})$ and applying the triangle
  inequality shows
  \begin{equation} \label{eq:lemma_laplacian_power4}
    \norm{\psi'_{t,n}} - \norm{\psi_{t,n}-\psi'_{t,n}} \leq \norm{\psi_{t,n}}
  \end{equation}
  We may directly calculate $\norm{\psi'_{t,n}} = t^K
  \frac{g^{(K)}(0)}{K!}  \norm{\L^K\delta_n}$, and we have
  $\norm{\psi_{t,n}-\psi'_{t,n}} \leq \sqrt{N} t^{K+1}
  \frac{\lambda_{N-1}^{K+1}}{(K+1)!} B$ from Lemma
  \ref{lemma:monomial_approx}. These imply together that the l.h.s. of
  (\ref{eq:lemma_laplacian_power4}) is greater than or equal to
  $t^K\left(\frac{g^{(K)}(0)}{K!}  \norm{\L^K\delta_n} - t
    \sqrt{N}\frac{\lambda_{N-1}^{K+1}}{(K+1)!} B\right)$. Together with
  (\ref{eq:lemma_laplacian_power3}), this shows

  \begin{equation}\label{eq:lemma_laplacian_power5}
    \frac{\psi_{t,n}(m)}{\norm{\psi_{t,n}}} \leq
    \frac{tC'}{a-tb}
  \end{equation}
  with $a=\frac{g^{(K)}(0)}{K!}  \norm{\L^K\delta_n}$ and
  $b=\sqrt{N}\frac{\lambda_{N-1}^{K+1}}{(K+1)!}B$. An elementary
  calculation shows $\frac{C't}{a-tb} \leq \frac{2C'}{a} t$ if $t\leq
  \frac{a}{2b}$. This implies the desired result with
  $D=\frac{2C'K!}{g^{(K)}(0)\norm{\L^K\delta_n}}$ and 
  $t''=\frac{g^{(K)}(0)\norm{\L^K\delta_n}(K+1)}{2\sqrt{N}\lambda_{N-1}^{K+1} B}$.
\end{proof}
\end{theorem}

\subsection{Spectral Graph Wavelet Frames}\label{sec:wavelet_frames}

The spectral graph wavelets depend on the continuous scale parameter
$t$. For any practical computation, $t$ must be sampled to a finite
number of scales. Choosing $J$ scales $\{t_j\}_{j=1}^J$ will yield a
collection of $NJ$ wavelets $\psi_{{t_j},n}$, along with the $N$
scaling functions $\scalf_n$.

It is a natural question to ask how well behaved this set of vectors
will be for representing functions on the vertices of the graph. We
will address this by considering the wavelets at discretized scales as
a frame, and examining the resulting frame bounds.

We will review the basic definition of a frame. A more complete
discussion of frame theory may be found in \cite{Daubechies1992} and
\cite{Heil1989}. Given a Hilbert space $\H$, a set of vectors
$\Gamma_k\in\H$ form a frame with frame bounds $A$ and $B$ if the inequality
\begin{equation}
  A \norm{f}^2 \leq \sum_k |\ip{f}{\Gamma_k}|^2 \leq B \norm{f}^2
\end{equation}
holds for all $f\in \H$.


The frame bounds $A$ and $B$ provide information about the numerical
stability of recovering the vector $f$ from inner product measurements
$\ip{f}{\Gamma_k}$. These correspond to the scaling function
coefficients $S_f(n)$ and wavelet coefficients $W_f(t_j,n)$ for the
frame consisting of the scaling functions and the spectral graph
wavelets with sampled scales.  As we shall see later in section
\ref{sec:rec}, the speed of convergence of algorithms used to invert
the spectral graph wavelet transform will depend on the frame bounds.


\begin{theorem}\label{theorem:sgw_frame}
  Given a set of scales $\{t_j\}_{j=1}^J$, the set $F=\{\scalf_{n}\}_{n=1}^N \cup
  \{\psi_{t_j,n}\}_{j=1}^J~_{n=1}^{N}$ forms a frame with bounds $A$,
  $B$ given by
  \begin{align}
    &A = \min_{\lambda\in[0,\lambda_{N-1}]} G(\lambda) \\ \nt
    &B = \max_{\lambda\in[0,\lambda_{N-1}]} G(\lambda),  \\ \nt
  \end{align}\label{eq:sgw_frame1}
  where $G(\lambda)=h^2(\lambda)+\sum_j g(t_j \lambda)^2$.
\begin{proof}
  Fix $f$. Using expression (\ref{eq:w_fourier_exp}), we see
  \begin{align} \label{eq:sgw_frame2}
  \sum_n|W_f(t,n)|^2 &=\sum_n \sum_\l g(t\lambda_\l)\evec_\l(n) \hat{f}(\l) 
  \sum_{\l'} \overline{g(t\lambda_{\l'})  \evec_{\l'}(n)  \hat{f}(\l')} \\ \nt
   &= \sum_\l |g(t\lambda_\l)|^2 |\hat{f}(\l)|^2 \nt
  \end{align}
  upon rearrangement and using $\sum_n \evec_\l(n) \overline{\evec_{\l'}(n)}
  = \delta_{\l,\l'}$. Similarly, 
\begin{equation} \label{eq:sgw_frame3}
  \sum_n |S_f(n)|^2 = \sum_\l |h(\lambda_\l)|^2 |\hat{f}(\l)|^2
\end{equation}
Denote by $Q$ the sum of squares of inner products of $f$ with vectors
in the collection $F$. Using (\ref{eq:sgw_frame2}) and
(\ref{eq:sgw_frame3}), we have
  \begin{align}\label{eq:sgw_frame4}
    Q &= \sum_\l \left(
      |h(\lambda_\l)|^2  + 
      \sum_{j=1}^J |g(t_j\lambda_\l)|^2 
      \right) |\hat{f}(\l)|^2 = \sum_\l G(\lambda_\l) |\hat{f}(\lambda_\l)|^2
  \end{align}
Then by the definition of $A$ and $B$, we have
\begin{equation}
A \sum_{\l=0}^{N-1} |\hat{f}(\l)|^2 \leq Q \leq B \sum_{\l=0}^{N-1} |\hat{f}(\l)|^2
\end{equation}
Using the Parseval relation $\norm{f}^2 = \sum_\l |\hat{f}(\l)|^2$ then
gives the desired result. 
\end{proof}
\end{theorem}

\section{Polynomial Approximation and Fast SGWT}
\label{sec:polyapprox}

We have defined the SGWT explicitly in the space of eigenfunctions of
the graph Laplacian. The naive way of computing the transform, by
directly using equation (\ref{eq:w_fourier_exp}), requires explicit
computation of the entire set of eigenvectors and eigenvalues of
$\L$. This approach scales poorly for large graphs.  General purpose
eigenvalue routines such as the QR algorithm have computational
complexity of $O(N^3)$ and require $O(N^2)$ memory
\cite{Watkins2007}. Direct computation of the SGWT through
diagonalizing $\L$ is feasible only for graphs with fewer than a few
thousand vertices. In contrast, problems in signal and image
processing routinely involve data with hundreds of thousands or
millions of dimensions. Clearly, a fast transform that avoids the need
for computing the complete spectrum of $\L$ is needed for the SGWT to
be a useful tool for practical computational problems.

We present a fast algorithm for computing the SGWT that is based on
approximating the scaled generating kernels $g$ by low order
polynomials. Given this approximation, the wavelet coefficients at
each scale can then be computed as a polynomial of $\L$ applied to the
input data. These can be calculated in a way that accesses $\L$ only
through repeated matrix-vector multiplication. This results in an
efficient algorithm in the important case when the graph is sparse,
i.e. contains a small number of edges.

We first show that the polynomial approximation may be taken over a
finite range containing the spectrum of $\L$.
\begin{lemma} \label{lemma:polyapprox} 
  Let $\lambda_{max}\geq \lambda_{N-1}$ be any upper bound on the
  spectrum of $\L$. For fixed $t>0$, let $p(x)$ be a polynomial
  approximant of $g(tx)$ with $L_{\infty}$ error
  $B=\sup_{x\in[0,\lambda_{max}]}|g(tx)-p(x)|$. 
  Then the approximate wavelet coefficients $\tilde{W}_f(t,n) =
  \left( p(\L) f \right)_n$ satisfy
  \begin{equation}
    |W_f(t,n)-\tilde{W}_f(t,n)| \leq B \norm{f}
  \end{equation}
  \begin{proof}
    Using equation (\ref{eq:w_fourier_exp}) we have
    \begin{align}
      |W_f(t,n)-\tilde{W}_f(t,n)| & = 
      \left|\sum_\l g(t\lambda_\l) \hat{f}(\l)\evec_\l(n) - 
       \sum_\l p(\lambda_\l) \hat{f}(\l) \evec_\l(n)
     \right| \nt \\
     & \leq \sum_{l} |g(t\lambda_\l)-p(\lambda_\l)| | \hat{f}(\l)
     \evec_\l(n)| \label{eq1:lemma:polyapprox} \\
     & \leq B \norm{f} 
     \end{align}
     The last step follows from using Cauchy-Schwartz and the
     orthonormality of the $\evec_\l$'s.
  \end{proof}
 \end{lemma}

 {\bf Remark } : The results of the lemma hold for any
 $\lambda_{max}\geq \lambda_{N-1}$. Computing extremal eigenvalues of
 a self-adjoint operator is a well studied problem, and efficient
 algorithms exist that access $\L$ only through matrix-vector
 multiplication, notably Arnoldi iteration or the Jacobi-Davidson
 method \cite{Watkins2007,Sleijpen1996}. In particular, good estimates
 for $\lambda_{N-1}$ may be computed at far smaller cost than
 that of computing the entire spectrum of $\L$.

 For fixed polynomial degree $M$, the upper bound on the approximation
 error from Lemma \ref{lemma:polyapprox} will be minimized if $p$ is
 the minimax polynomial of degree $M$ on the interval
 $[0,\lambda_{max}]$. Minimax polynomial approximations are well
 known, in particular it has been shown that they exist and are unique
 \cite{Cheney1966}. Several algorithms exist for computing minimax
 polynomials, most notably the Remez exchange algorithm
 \cite{Fraser1965}.

 \begin{figure}
   \begin{tabular}{c @{} c}
   \includegraphics[width=.47\columnwidth]{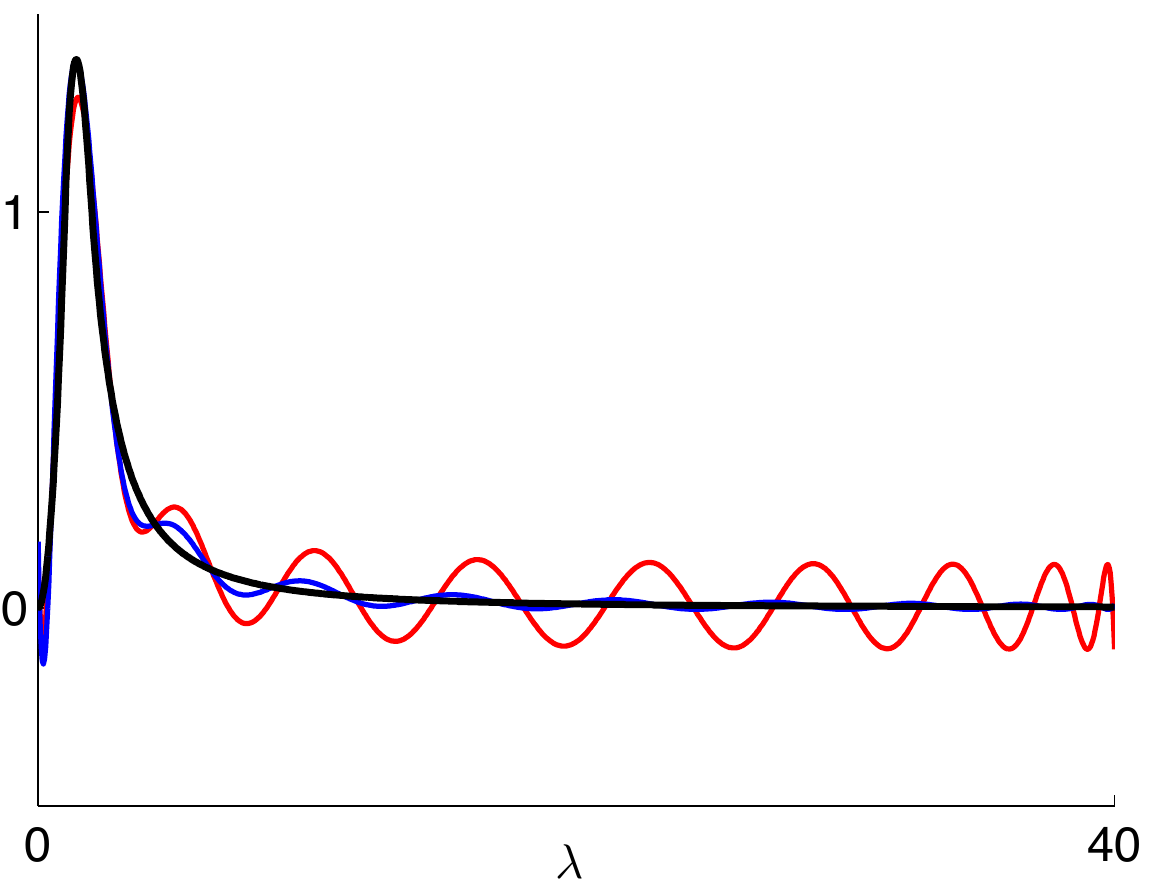} &
   \includegraphics[width=.47\columnwidth]{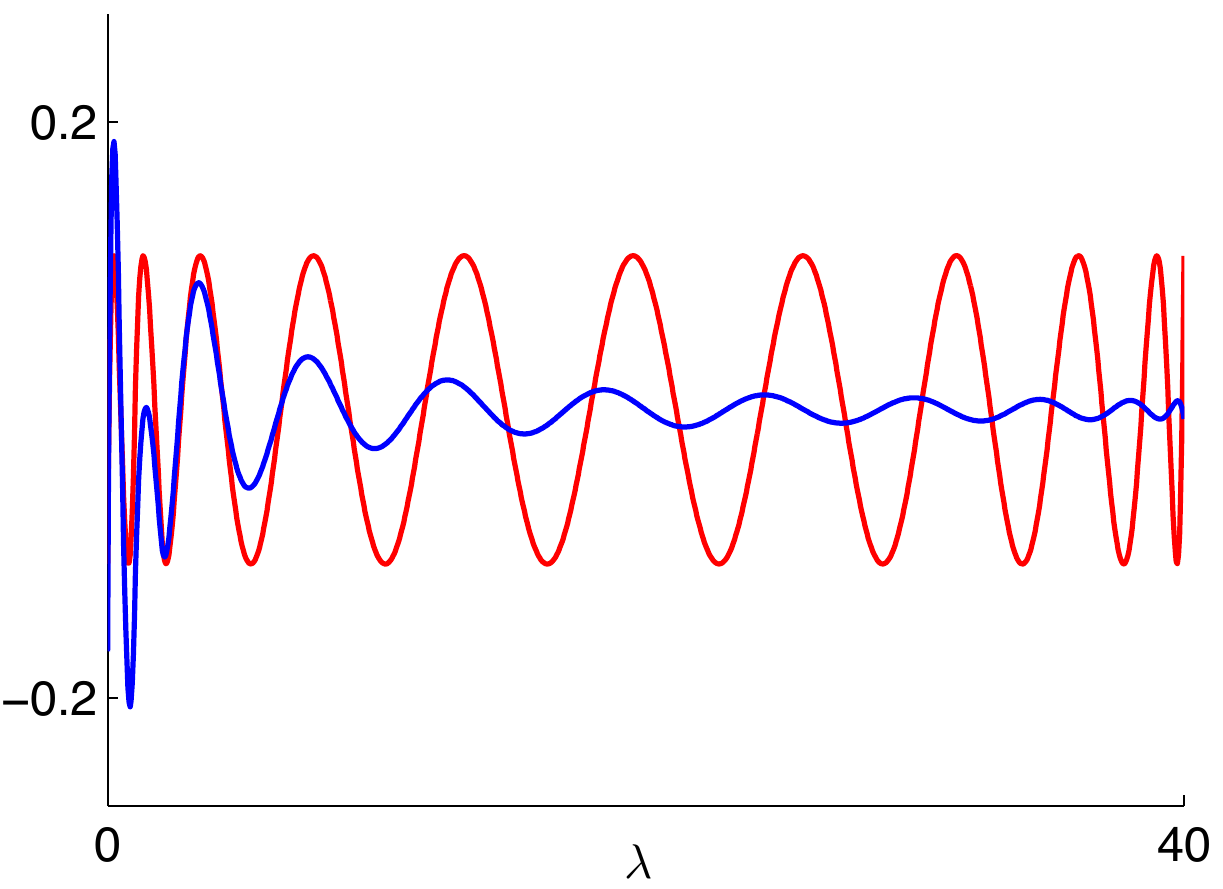}\\
   (a) & (b) \\
   \end{tabular}
   \caption{(a) Wavelet kernel $g(\lambda)$ (black), truncated Chebyshev
     approximation (blue) and minimax polynomial approximation (red)
     for degree $m=20$.  Approximation errors shown in (b), truncated
     Chebyshev polynomial has maximum error 0.206, minimax polynomial
     has maximum error 0.107 .}\label{fig:cheby_vs_minimax}
\end{figure}

In this work, however, we will instead use a polynomial approximation
given by the truncated Chebyshev polynomial expansion of $g(tx)$. It
has been shown that for analytic functions in an ellipse containing
the approximation interval, the truncated Chebyshev expansions gives
an approximate minimax polynomial \cite{Geddes1978}.  Minimax
polynomials of order $m$ are distinguished by having their
approximation error reach the same extremal value at $m+2$ points in
their domain. As such, they distribute their approximation error
across the entire interval. We have observed that for the wavelet
kernels we use in this work, truncated Chebyshev polynomials result in
a minimax error only slightly higher than the true minimax
polynomials, and have a much lower approximation error where the
wavelet kernel to be approximated is smoothly varying. A
representative example of this is shown in Figure
\ref{fig:cheby_vs_minimax}. We have observed that for small weighted
graphs where the wavelets may be computed directly in the spectral
domain, the truncated Chebyshev approximations give slightly lower
approximation error than the minimax polynomial approximations
computed with the Remez algorithm. 
 
 For these reasons, we use approximating polynomials given by
 truncated Chebyshev polynomials. In addition, we will exploit the
 recurrence properties of the Chebyshev polynomials for efficient
 evaluation of the approximate wavelet coefficients.  An overview of
 Chebyshev polynomial approximation may be found in
 \cite{Phillips2003}, we recall here briefly a few of their key
 properties.
 
 The Chebyshev polynomials $T_k(y)$ may be generated by the stable
 recurrence relation $T_k(y) = 2 y T_{k-1}(y) - T_{k-2}(y)$, with $T_0
 = 1$ and $T_1=y$. For $y\in[-1,1]$, they satisfy the trigonometric
 expression $T_k(y) =\cos\left(k\arccos(y)\right)$, which shows that
 each $T_k(y)$ is bounded between -1 and 1 for $y\in[-1,1]$. The
 Chebyshev polynomials form an orthogonal basis for
 $L^2([-1,1],\frac{dy}{\sqrt{1-y^2}})$, the Hilbert space of square
 integrable functions with respect to the measure
 $dy/\sqrt{1-y^2}$. In particular they satisfy
 \begin{equation}
   \int_{-1}^1 \frac{T_l(y) T_m(y)}{\sqrt{1-y^2}} dy = 
   \begin{cases}
     \delta_{l,m} \pi/2 & \mbox { if $m,l>0$ }\\
     \pi & \mbox { if $m=l=0$ }
   \end{cases}
 \end{equation}
 Every $h\in L^2([-1,1],\frac{dy}{\sqrt{1-y^2}})$ has a uniformly
 convergent Chebyshev series 
 \begin{equation}\label{eq:cheby_expansion}
   h(y) = \frac{1}{2} c_0 + \sum_{k=1}^{\infty} c_k T_k(y)
 \end{equation}
 with Chebyshev coefficients
 \begin{equation}
   c_k = \frac{2}{\pi}\int_{-1}^1 \frac{T_k(y) h(y)}{\sqrt{1-y^2}} dy = 
   \frac{2}{\pi} \int_0^\pi \cos(k\theta) h(\cos(\theta)) d\theta
 \end{equation}
 We now assume a fixed set of wavelet scales $t_n$.  For each $n$,
 approximating $g(t_nx)$ for $x\in[0,\lambda_{max}]$ can be done by
 shifting the domain using the transformation $x=a(y+1)$, with
 $a=\lambda_{max}/2$. Denote the shifted Chebyshev polynomials
 $\overline{T}_k(x) = T_k(\frac{x-a}{a})$. We may then write
 \begin{equation} \label{eq:gtn-cheby-expansion}
   g(t_nx)=\frac{1}{2} c_{n,0} +\sum_{k=1}^{\infty} c_{n,k} \overline{T}_k(x),
 \end{equation}
 valid for $x\in[0,\lambda_{max}]$, with
 \begin{equation}
c_{n,k}=\frac{2}{\pi} \int_0^\pi
 \cos(k\theta)g(t_n(a (\cos(\theta)+1) )) d\theta.
\end{equation}

For each scale $t_j$, the approximating polynomial $p_j$ is achieved
by truncating the Chebyshev expansion (\ref{eq:gtn-cheby-expansion})
to $M_j$ terms. We may use exactly the same scheme to approximate the
scaling function kernel $h$ by the polynomial $p_0$.

Selection of the values of $M_j$ may be considered a design problem,
posing a trade-off between accuracy and computational cost. The fast
SGWT approximate wavelet and scaling function coefficients are then
given by
 \begin{align} \label{eq:fast-sgwt}
   \tilde{W}_f(t_j,n) = \left( \frac{1}{2}c_{j,0} f + \sum_{k=1}^{M_j}
   c_{j,k} \overline{T}_k(\L) f \right)_n \nt \\
  \tilde{S}_f(n) = \left( \frac{1}{2}c_{0,0} f + \sum_{k=1}^{M_0}
   c_{0,k} \overline{T}_k(\L) f \right)_n \nt \\
\end{align}

The utility of this approach relies on the efficient computation of
$\overline{T}_k(\L)f$. Crucially, we may use the Chebyshev recurrence
to compute this for each $k<M_j$ accessing $\L$ only through
matrix-vector multiplication. As the shifted Chebyshev polynomials
satisfy $\overline{T}_k(x) = \frac{2}{a}(x-1)\overline{T}_{k-1}(x)
- \overline{T}_{k-2}(x)$, we have for any $f\in \R^N$,
\begin{equation} \label{eq:cheby-f-recurrence}
  \overline{T}_k(\L)f = 
  \frac{2}{a}(\L- I)\left( \overline{T}_{k-1}(\L) f\right)
  - \overline{T}_{k-2}(\L)f
\end{equation}

Treating each vector $\overline{T}_k(\L)f$ as a single symbol, this
relation shows that the vector $\overline{T}_k(\L)f$ can be computed
from the vectors $\overline{T}_{k-1}(\L)f$ and
$\overline{T}_{k-2}(\L)f$ with computational cost dominated by a
single matrix-vector multiplication by $\L$.

Many weighted graphs of interest are sparse, i.e. they have a small
number of nonzero edges. Using a sparse matrix representation, the
computational cost of applying $\L$ to a vector is proportional to
$|E|$, the number of nonzero edges in the graph.  The computational
complexity of computing all of the Chebyshev polynomials $T_k(\L)f$
for $k\leq M$ is thus $O(M |E|)$.  The scaling function and wavelet
coefficients at different scales are formed from the same set of
$T_k(\L)f$, but by combining them with different coefficients
$c_{j,k}$. The computation of the Chebyshev polynomials thus need not
be repeated, instead the coefficients for each scale may be computed
by accumulating each term of the form $c_{j,k} T_k(\L)f$ as $T_k(\L)f$
is computed for each $k\leq M$. This requires $O(N)$ operations at
 scale $j$ for each $k\leq M_j$, giving an overall computational
complexity for the fast SGWT of $O(M |E| + N \sum_{j=0}^J M_j )$,
where $J$ is the number of wavelet scales.  In particular, for classes
of graphs where $|E|$ scales linearly with $N$, such as graphs of
bounded maximal degree, the fast SGWT has computational complexity
$O(N)$. Note that if the complexity is dominated by the computation of
the $T_k(\L)f$, there is little benefit to choosing $M_j$ to vary with
$j$.

Applying the recurrence (\ref{eq:cheby-f-recurrence}) requires memory
of size $3N$.  The total memory requirement for a straightforward
implementation of the fast SGWT would then be $N(J+1)+3N$.


\subsection{Fast computation of Adjoint}

Given a fixed set of wavelet scales $\{t_j\}_{j=1}^J$, and including
the scaling functions $\phi_n$, one may consider the overall wavelet
transform as a linear map $W:\R^N \to \R^{N(J+1)}$ defined by
$Wf=\left( (T_hf)^T,(T_g^{t_1}f)^T,\cdots,(T_g^{t_J}f)^T\right)^T$.
Let $\tilde{W}$ be the corresponding approximate wavelet transform
defined by using the fast SGWT approximation, i.e.
$\tilde{W}f=\left((p_0(\L)f)^T,(p_1(\L)f)^T,\cdots,(p_J(\L)f)^T\right)^T$.
We show that both the adjoint $\tilde{W}^*:\R^{N(J+1)}\to \R^N$, and
the composition $W^*W:R^N\to R^N$ can be computed efficiently using
Chebyshev polynomial approximation. This is important as several
methods for inverting the wavelet transform or using the spectral
graph wavelets for regularization can be formulated using the adjoint
operator, as we shall see in detail later in Section \ref{sec:rec}.

For any $\eta\in \R^{N(J+1)}$, we consider $\eta$ as the concatenation
$\eta=(\eta_0^T,\eta_1^T,\cdots,\eta_J^T)^T$ with each $\eta_j \in
\R^N$ for $0 \leq j \leq J$. Each $\eta_j$ for $j\geq 1$ may be thought of as
a subband corresponding to the scale $t_j$, with $\eta_0$ representing
the scaling function coefficients. We then have
\begin{align} \label{eq:adjoint}
\ip{\eta}{Wf}_{N(J+1)} &= \ip{\eta_0}{T_h f} +
\sum_{j=1}^J \ip{\eta_j}{T_g^{t_j} f}_N \nt \\
& = \ip{T_h^* \eta_0}{f} + 
\left<\sum_{j=1}^J (T_g^{t_j})^*\eta_j,f\right>_N 
 = \left<T_h \eta_0+ \sum_{j=1}^J T_g^{t_j} \eta_j,f\right>_N
\end{align}
as $T_h$ and each $T_g^{t_j}$ are self adjoint.  As (\ref{eq:adjoint})
holds for all $f\in \R^N$, it follows that $W^*\eta = T_h \eta_0
+\sum_{j=1}^J T_g^{t_j} \eta_n$, i.e. the adjoint is given by
re-applying the corresponding wavelet or scaling function operator on
each subband, and summing over all scales.

This can be computed using the same fast Chebyshev polynomial
approximation scheme in equation (\ref{eq:fast-sgwt}) as for the
forward transform, e.g. as $\tilde{W}^*\eta = \sum_{j=0}^J p_j(\L) \eta_j$.
Note that this scheme computes the {\em exact} adjoint of the
approximate forward transform, as may be verified by replacing
$T_h$ by $p_0(\L)$ and $T_g^{t_j}$ by $p_j(\L)$ in (\ref{eq:adjoint}).

We may also develop a polynomial scheme for computing $\tilde{W}^*
\tilde{W}$. Naively computing this by first applying $\tilde{W}$, then
$\tilde{W}^*$ by the fast SGWT would involve computing $2J$ Chebyshev
polynomial expansions. By precomputing the addition of squares of the
approximating polynomials, this may be reduced to application of a
single Chebyshev polynomial with twice the degree, reducing the
computational cost by a factor $J$.  Note first that
\begin{align}
  {\tilde{W}}^* \tilde{W} f = \sum_{j=0}^J p_j(\L) \left(p_j(\L) f \right) 
  =\left( \sum_{j=0}^J (p_j(\L))^2 \right)f
\end{align}
Set $P(x)=\sum_{j=0}^J (p_j(x))^2$, which has degree
$M^*=2\max\{M_j\}$.  We seek to express $P$ in the shifted Chebyshev
basis as $P(x)=\frac{1}{2}d_0 + \sum_{k=1}^{M^*} d_k
\overline{T}_k(x)$. The Chebyshev polynomials satisfy the product
formula 
\begin{equation}\label{eq:cheby_product}
T_k(x) T_l(x) = \frac{1}{2}\left(T_{k+l}(x) + T_{|k-l|}(x) \right)
\end{equation}
which we will use to compute the Chebyshev coefficients $d_k$ in terms
of the Chebyshev coefficients $c_{j,k}$ for the individual $p_j$'s.

Expressing this explicitly is slightly complicated by the convention
that the $k=0$ Chebyshev coefficient is divided by $2$ in the
Chebyshev expansion (\ref{eq:gtn-cheby-expansion}). For convenience in
the following, set $c'_{j,k}=c_{j,k}$ for $k \geq 1$ and
$c'_{j,0}=\frac{1}{2}c_{j,0}$, so that $p_j(x) = \sum_{k=0}^{M_n}
c'_{j,k} \overline{T}_k(x)$.
Writing $(p_j(x))^2 = \sum_{k=0}^{2*M_n}d'_{j,k} \overline{T}_k(x)$,
and applying (\ref{eq:cheby_product}), we compute
\begin{equation}
d'_{j,k} = 
\begin{cases}
\frac{1}{2}\left( {c'_{j,0}}^2 + \sum_{i=0}^{M_n} {c'_{j,i}}^2\right) &
\mbox{ if $k=0$ } \\
\frac{1}{2}\left( 
\sum_{i=0}^k c'_{j,i} c'_{j,k-i} +
\sum_{i=0}^{M_j-k} c'_{j,i} c'_{j,k+i} +
\sum_{i=k}^{M_j} c'_{j,i} c'_{j,i-k}
\right) & \mbox{ if $0<k\leq M_j$ } \\
\frac{1}{2}\left( 
\sum_{i=k-M_j}^{M_j} c'_{j,i} c'_{j,k-i} 
\right) & \mbox { if $M_j < k \leq 2M_j$}
\end{cases}
\end{equation}

Finally, setting $d_{n,0} = 2d'_{j,0}$ and $d_{j,k}=d'_{j,k}$ for
$k\geq 1$, and setting $d_k = \sum_{j=0}^J d_{j,k}$ gives the
Chebyshev coefficients for $P(x)$. We may then compute
\begin{equation} \label{eq:fast-cheby-frameop}
  \tilde{W}^* \tilde{W} f = P(\L) f 
= \frac{1}{2} d_0 f + \sum_{k=1}^{M^*} d_k \overline{T}_k(\L) f
\end{equation}
following (\ref{eq:fast-sgwt}).

\section{Reconstruction}
\label{sec:rec}

For most interesting signal processing applications, merely
calculating the wavelet coefficients is not sufficient. A wide class
of signal processing applications are based on manipulating the
coefficients of a signal in a certain transform, and later inverting
the transform. For the SGWT to be useful for more than simply signal
analysis, it is important to be able to recover a signal
corresponding to a given set of coefficients.

The SGWT is an overcomplete transform as there are more wavelets
$\psi_{{t_j},n}$ than original vertices of the graph. Including the
scaling functions $\scalf_n$ in the wavelet frame, the SGWT maps an input
vector $f$ of size $N$ to the $N(J+1)$ coefficients $c=Wf$. As is well
known, this means that $W$ will have an infinite number of
left-inverses $M$ s.t. $MWf=f$. A natural choice among the possible
inverses is to use the pseudoinverse $L=(W^*W)^{-1} W^*$. The pseudoinverse 
satisfies the minimum-norm property
\begin{equation}
Lc =\argmin_{f\in \R^N} \norm{c-Wf}_2
\end{equation}
For applications which involve manipulation of the wavelet
coefficients, it is very likely to need to apply the inverse to a a
set of coefficients which no longer lie directly in the image of
$W$. The above property indicates that, in this case, the
pseudoinverse corresponds to orthogonal projection onto the image of
$W$, followed by inversion on the image of $W$.

Given a set of coefficients $c$, the pseudoinverse will be given by
solving the square matrix equation $(W^*W) f = W^*c$. This system is
too large to invert directly. Solving it may be performed using any of
a number of iterative methods, including the classical frame algorithm
\cite{Daubechies1992}, and the faster conjugate gradients method
\cite{Grochenig1993}. These methods have the property that each step
of the computation is dominated by application of $W^*W$ to a single
vector. We use the conjugate gradients method, employing the fast
polynomial approximation (\ref{eq:fast-cheby-frameop}) for computing
application of $\tilde{W}^*\tilde{W}$.

\section{Implementation and examples}
\label{sec:examp}

In this section we first give the explicit details of the wavelet and
scaling function kernels used, and how we select the scales. We then
show examples of the spectral graph wavelets on several different real
and synthetic data sets.

\subsection{SGWT design details} \label{sec:sgwt_design} 

Our choice for the wavelet generating kernel $g$ is motivated by the
desire to achieve localization in the limit of fine scales. According
to Theorem \ref{theorem:localization}, localization can be ensured if
$g$ behaves as a monic power of $x$ near the origin. We choose $g$ to
be exactly a monic power near the origin, and to have power law decay
for large $x$. In between, we set $g$ to be a cubic spline such that
$g$ and $g'$ are continuous. Our $g$ is parametrized by the integers
$\alpha$ and $\beta$, and $x_1$ and $x_2$ determining the transition
regions : 
\begin{equation} \label{eq:kernel_spec}
  g(x;\alpha,\beta,x_1,x_2)=
\begin{cases}
  x_1^{-\alpha} x^\alpha & \mbox{ for $x < x_1$} \\
  s(x) & \mbox{ for $x_1 \leq x \leq x_2$}\\
  x_2^{\beta}x^{-\beta} & \mbox{ for $x>x_2$ }    \\
\end{cases}
\end{equation}
Note that $g$ is normalized such that $g(x_1)=g(x_2)=1$.  The
coefficients of the cubic polynomial $s(x)$ are determined by the
continuity constraints $s(x_1)=s(x_2)=1$ , $s'(x_1) = \alpha/x_1$ and
$s'(x_2)=-\beta/x_2$. All of the examples in this paper were produced
using $\alpha=\beta=1$, $x_1=1$ and $x_2=2$; in this case
$s(x)=-5+11x-6x^2+x^3$.

The wavelet scales $t_j$ are selected to be logarithmically equispaced
between the minimum and maximum scales $t_J$ and $t_1$. These are
themselves adapted to the upper bound $\lambda_{max}$ of the spectrum
of $\L$. The placement of the maximum scale $t_1$ as well as the
scaling function kernel $h$ will be determined by the selection of
$\lambda_{min}=\lambda_{max}/K$, where $K$ is a design parameter of
the transform. We then set $t_1$ so that $g(t_1 x)$ has power-law
decay for $x>\lambda_{min}$, and set $t_J$ so that $g(t_J x)$ has
monic polynomial behaviour for $x<\lambda_{max}$. This is achieved by
$t_1 = x_2/\lambda_{min}$ and $t_J=x_2/\lambda_{max}$.

For the scaling function kernel we take $h(x)=\gamma
\exp(-(\frac{x}{0.6\lambda_{min}})^4)$, where $\gamma$ is set such
that $h(0)$ has the same value as the maximum value of $g$. 

This set of scaling function and wavelet generating kernels, for
parameters $\lambda_{max}=10$, $K=20$, $\alpha=\beta=2$, $x_1=1$,
$x_2=2$, and $J=4$, are shown in Figure \ref{fig:spectral_design}.

\subsection{Illustrative examples : spectral graph wavelet gallery}

\begin{figure} [t]
\begin{tabular}{c@{}c@{}c}
\includegraphics[width=.32\columnwidth]{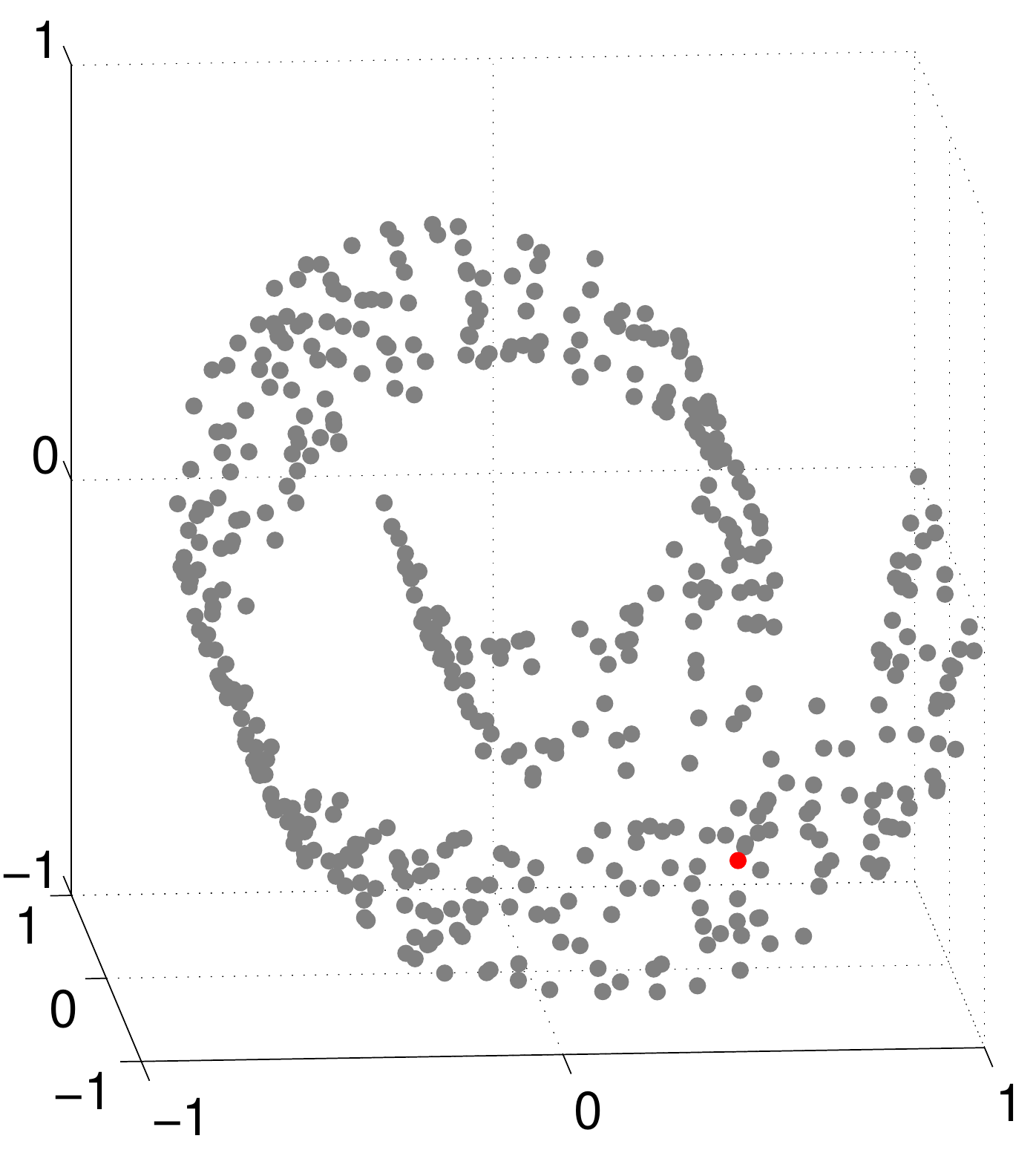} & 
\includegraphics[width=.32\columnwidth]{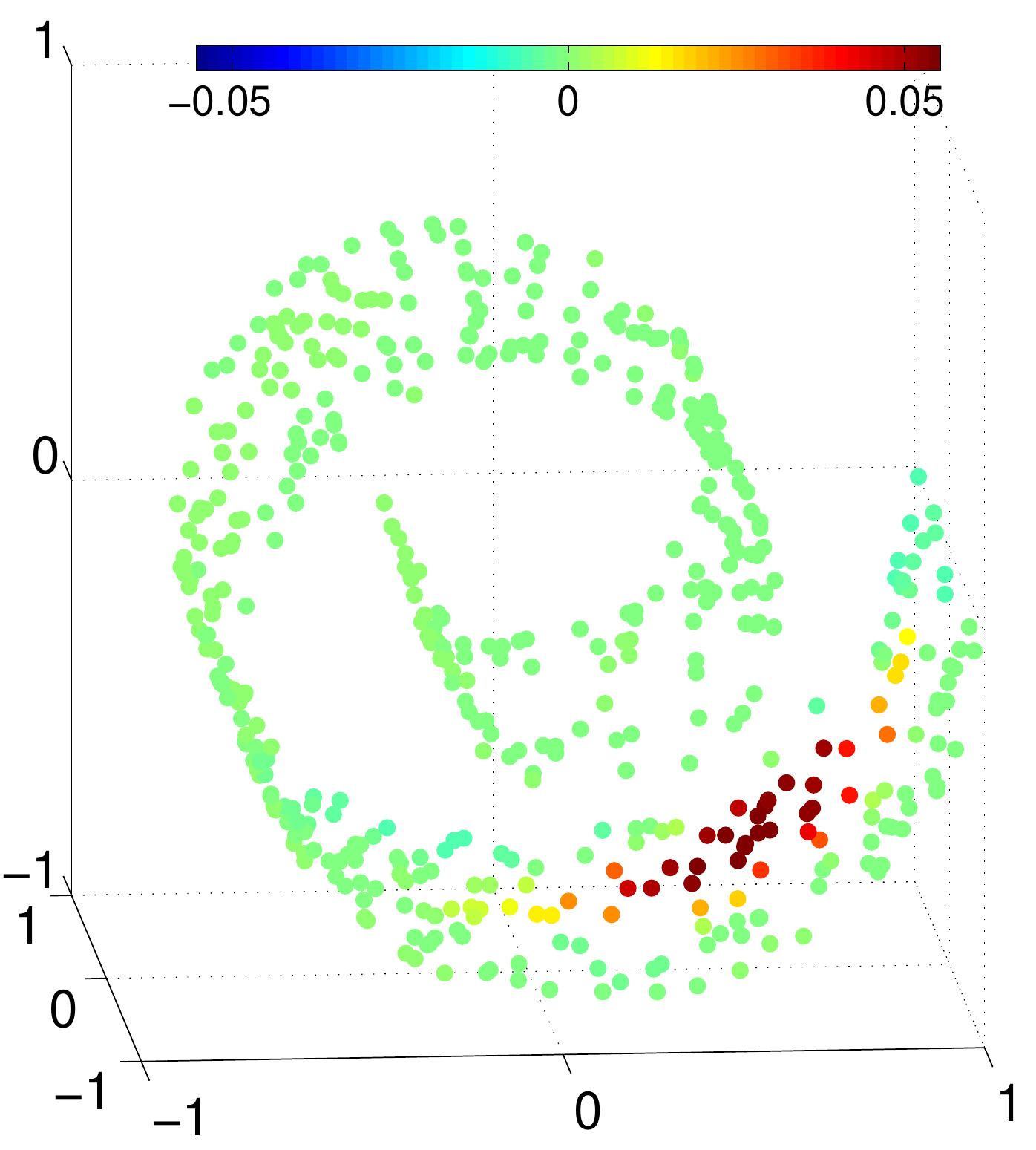} &
\includegraphics[width=.32\columnwidth]{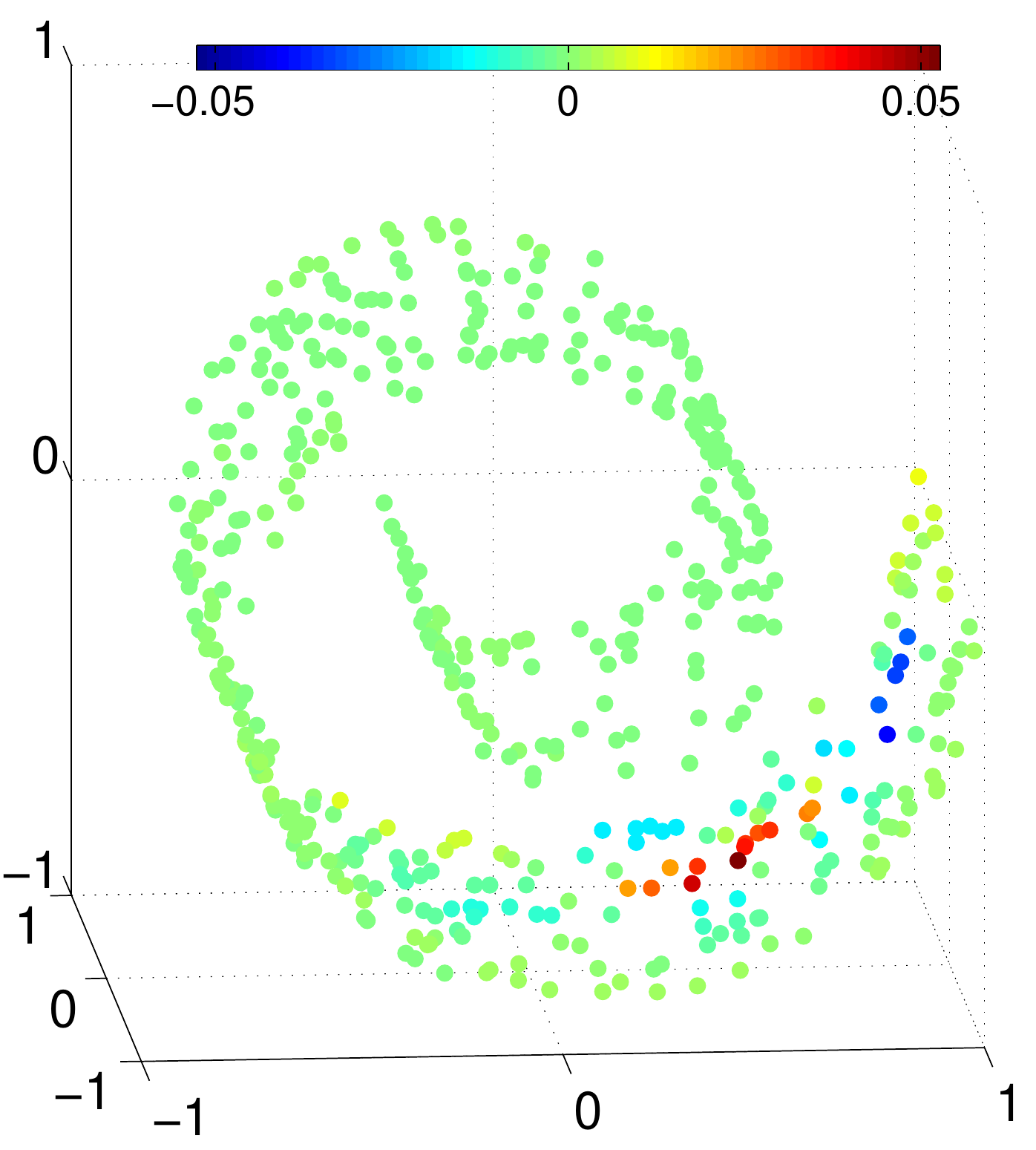} \\
(a) & (b) & (c) \\
\includegraphics[width=.32\columnwidth]{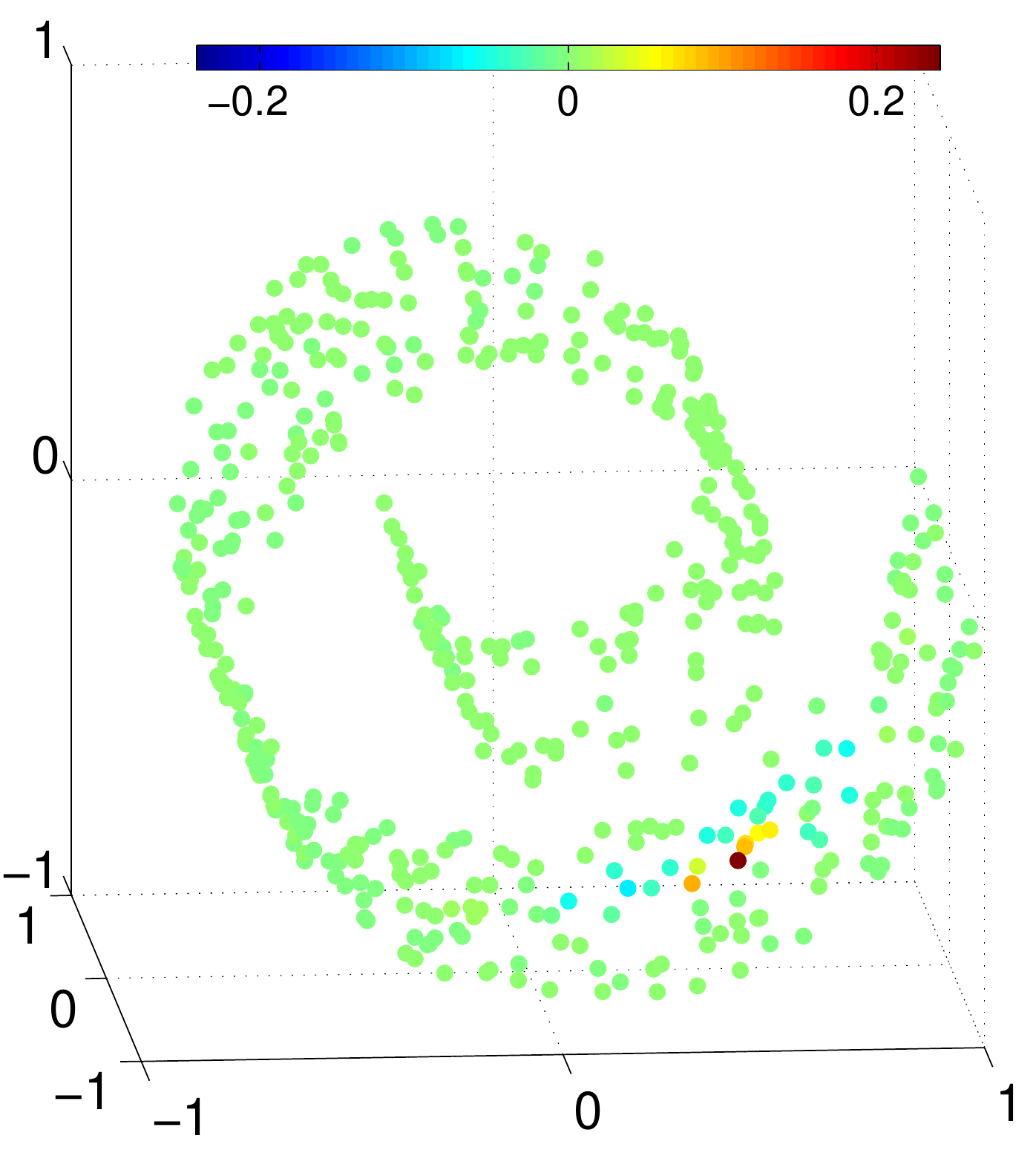} & 
\includegraphics[width=.32\columnwidth]{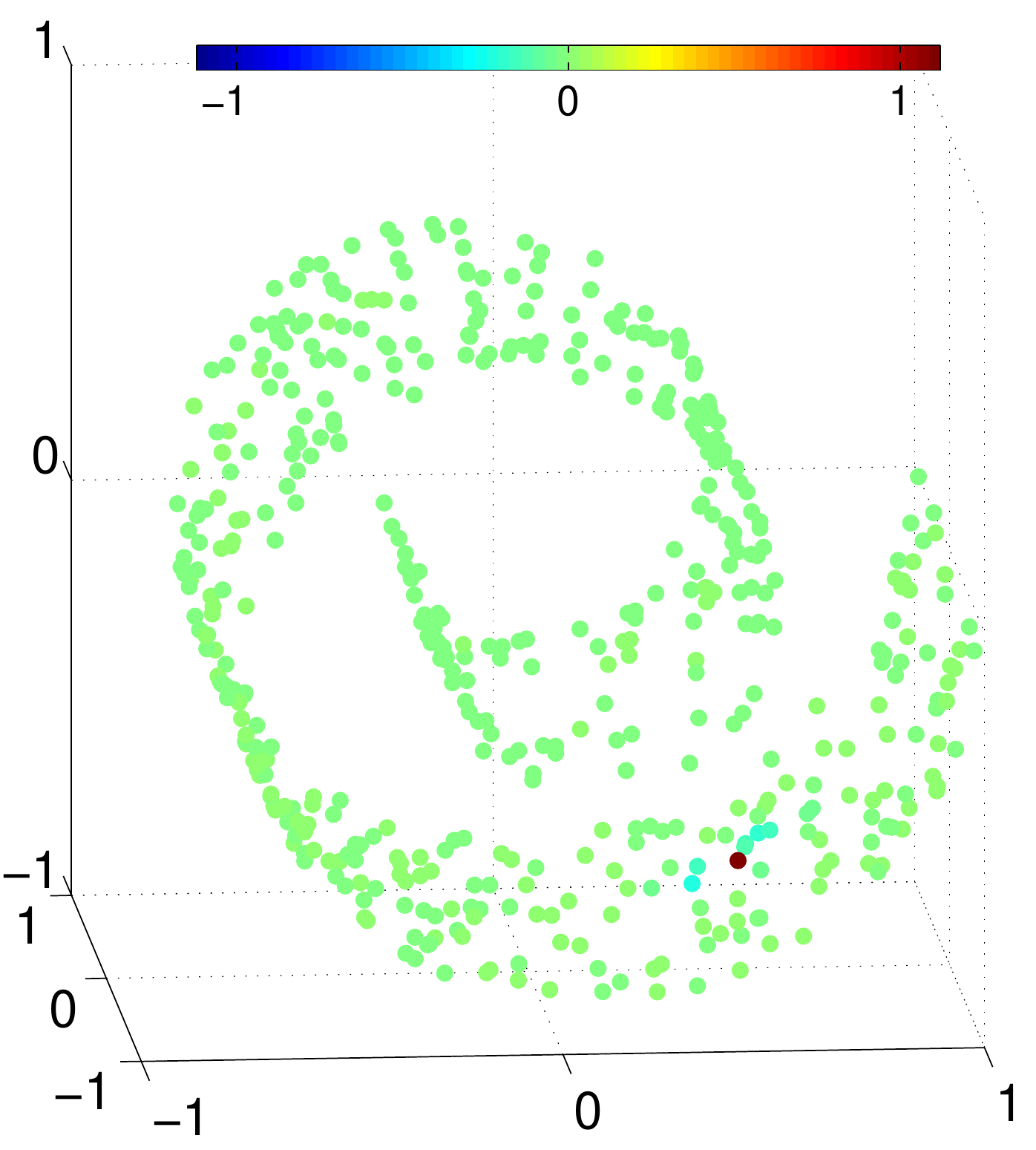} &
\includegraphics[width=.32\columnwidth]{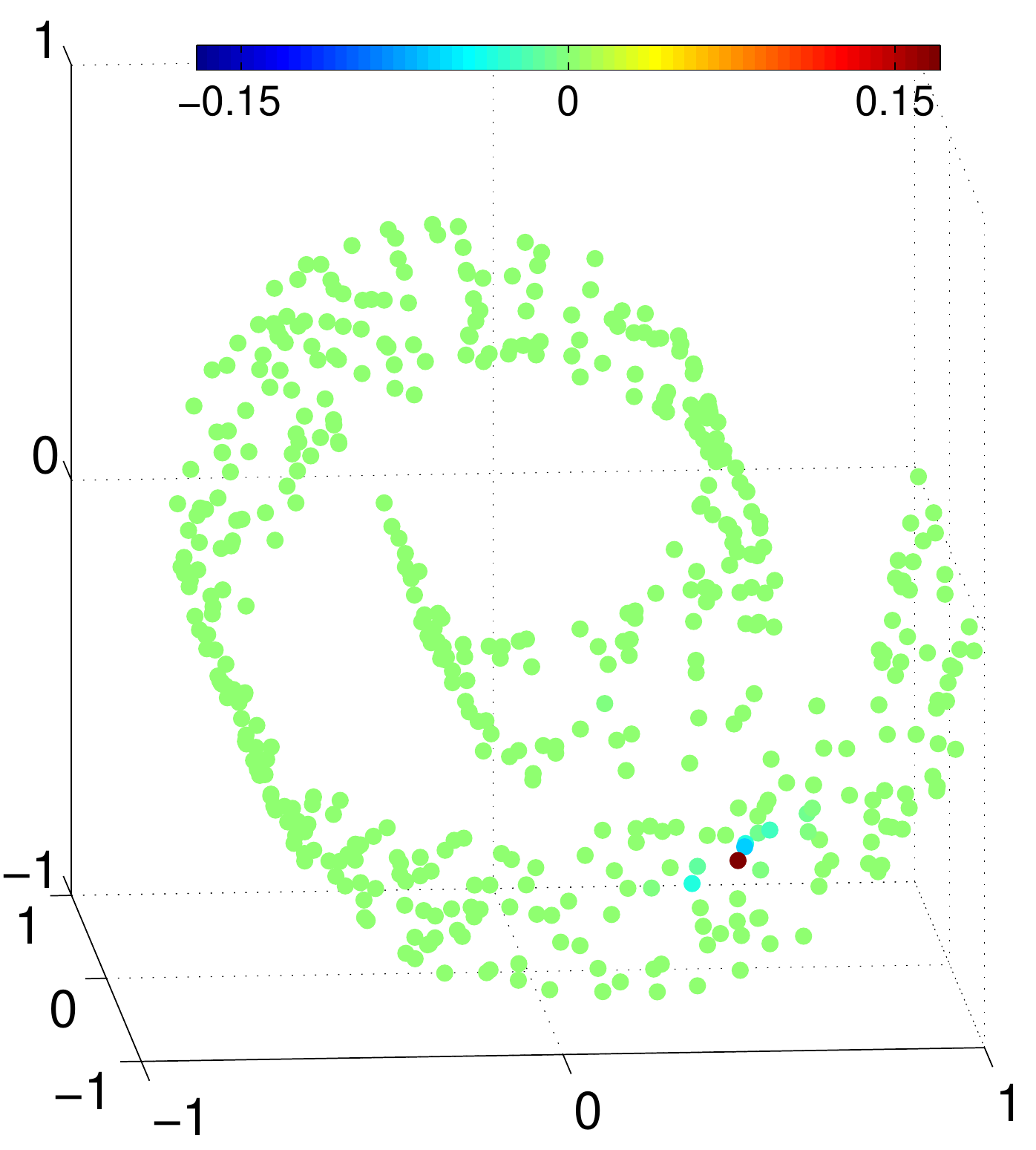} \\
(d) & (e) & (f) \\
\end{tabular}
\caption{Spectral graph wavelets on Swiss Roll data cloud, with $J=4$
  wavelet scales.  (a) vertex at which wavelets are centered (b)
  scaling function (c)-(f) wavelets, scales 1-4.}
\label{fig:swissroll}
\end{figure}

As a first example of building wavelets in a point cloud domain, we
consider the spectral graph wavelets constructed on the ``Swiss
roll''.  This example data set consists of points randomly sampled on
a 2-d manifold that is embedded in $\R^3$. The manifold is described
parametrically by $ \vec{x}(s,t) =(t\cos(t)/4\pi,s,t\sin(t)/4\pi)$ for
$-1\leq s \leq 1$, $\pi\leq t \leq 4\pi$. For our example we take 500
points sampled uniformly on the manifold.

Given a collection $x_i$ of points, we build a weighted graph by
setting edge weights $w_{i,j} =
\exp(-\norm{x_j-x_j}^2/2\sigma^2)$. For larger data sets this graph
could be sparsified by thresholding the edge weights, however we do
not perform this here. In Figure \ref{fig:swissroll} we show the Swiss
roll data set, and the spectral graph wavelets at four different
scales localized at the same location. We used $\sigma=0.1$ for
computing the underlying weighted graph, and $J=4$ scales with $K=20$
for computing the spectral graph wavelets.  In many examples relevant
for machine learning, data are given in a high dimensional space that
intrinsically lie on some underlying lower dimensional manifold. This
figure shows how the spectral graph wavelets can implicitly adapt to
the underlying manifold structure of the data, in particular notice
that the support of the coarse scale wavelets diffuse locally along
the manifold and do not ``jump'' to the upper portion of the roll.

\begin{figure}[t]
\begin{tabular}{c@{}c@{}c}
\includegraphics[width=.32\columnwidth]{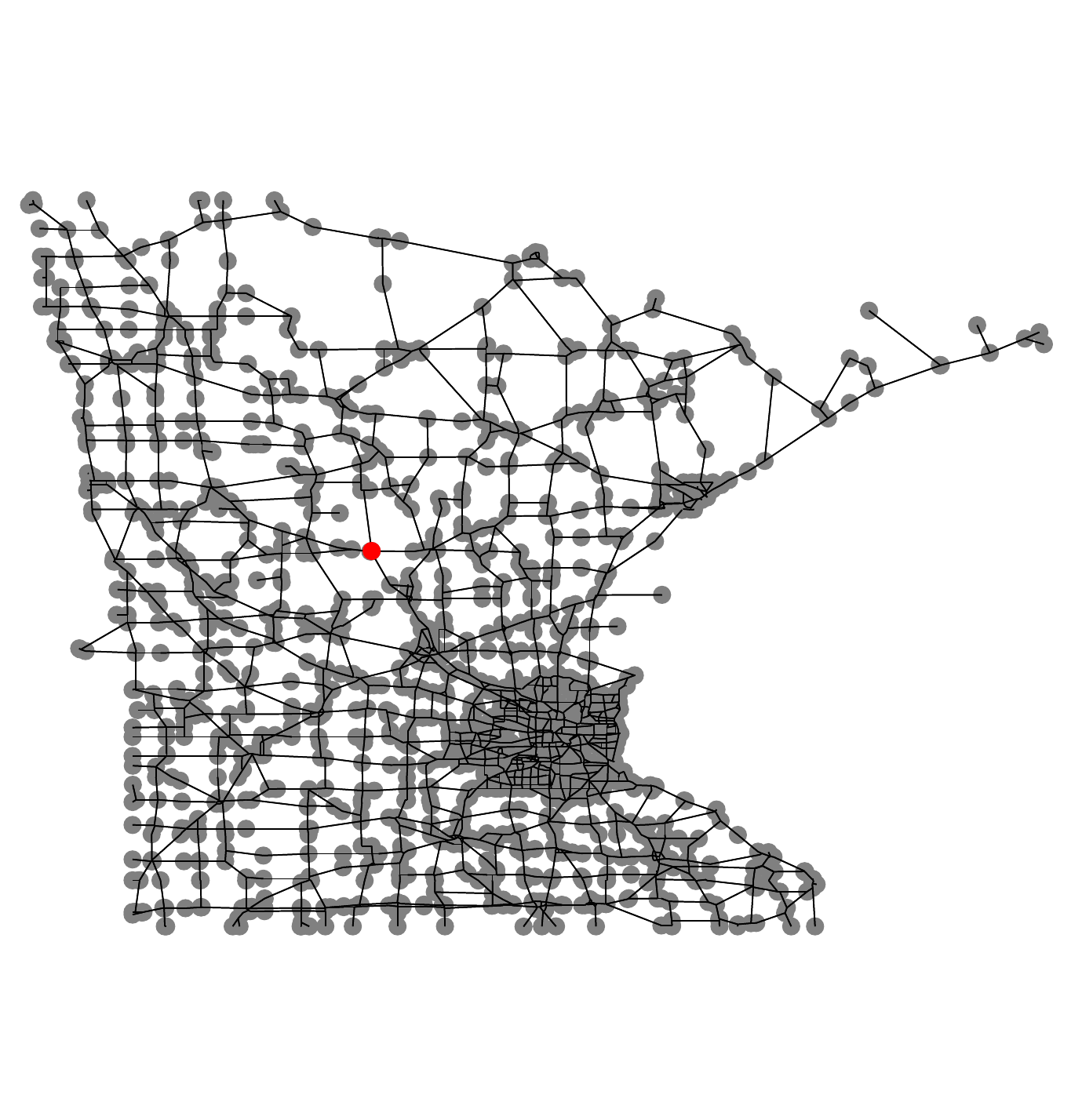} & 
\includegraphics[width=.32\columnwidth]{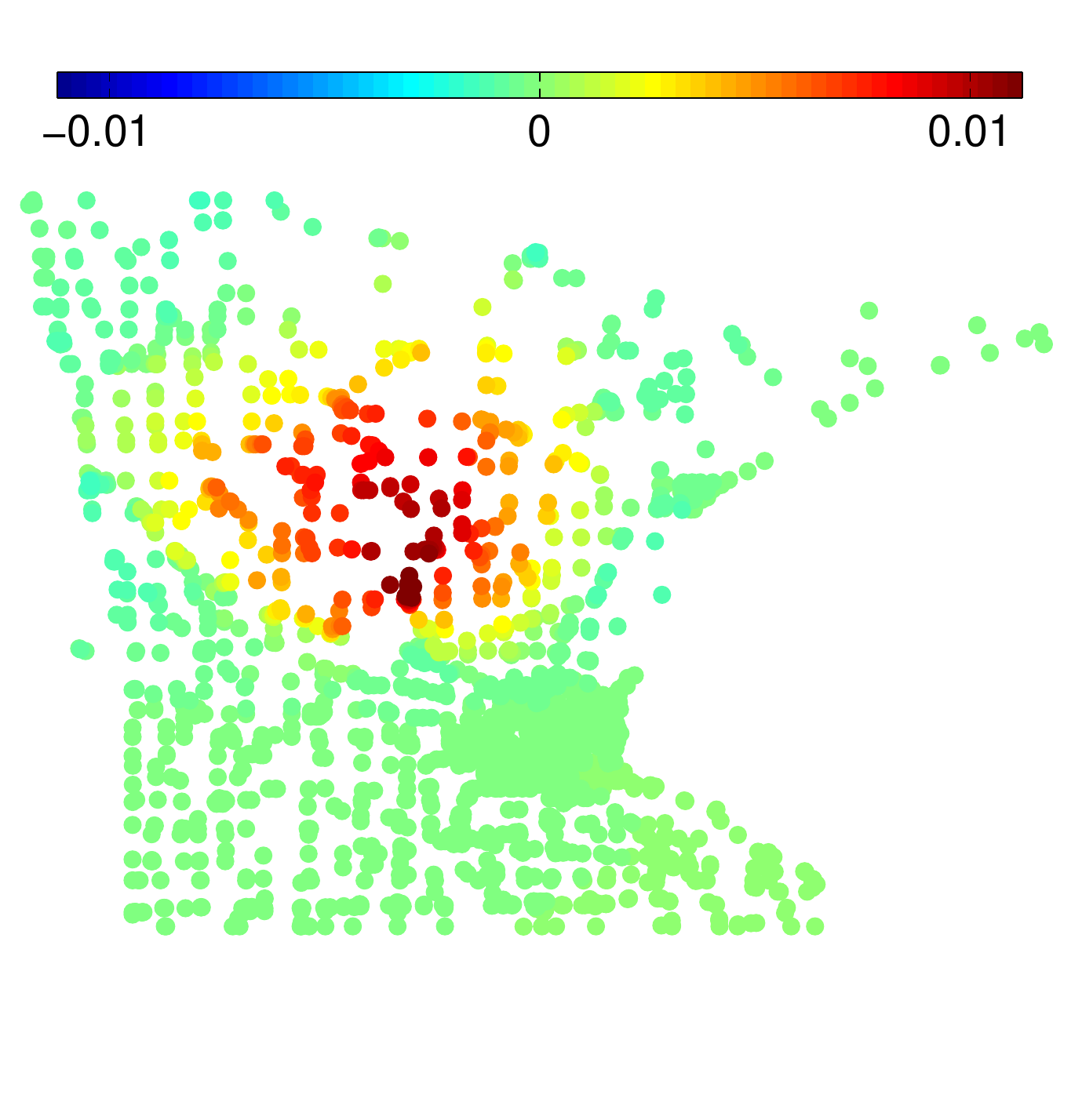} &
\includegraphics[width=.32\columnwidth]{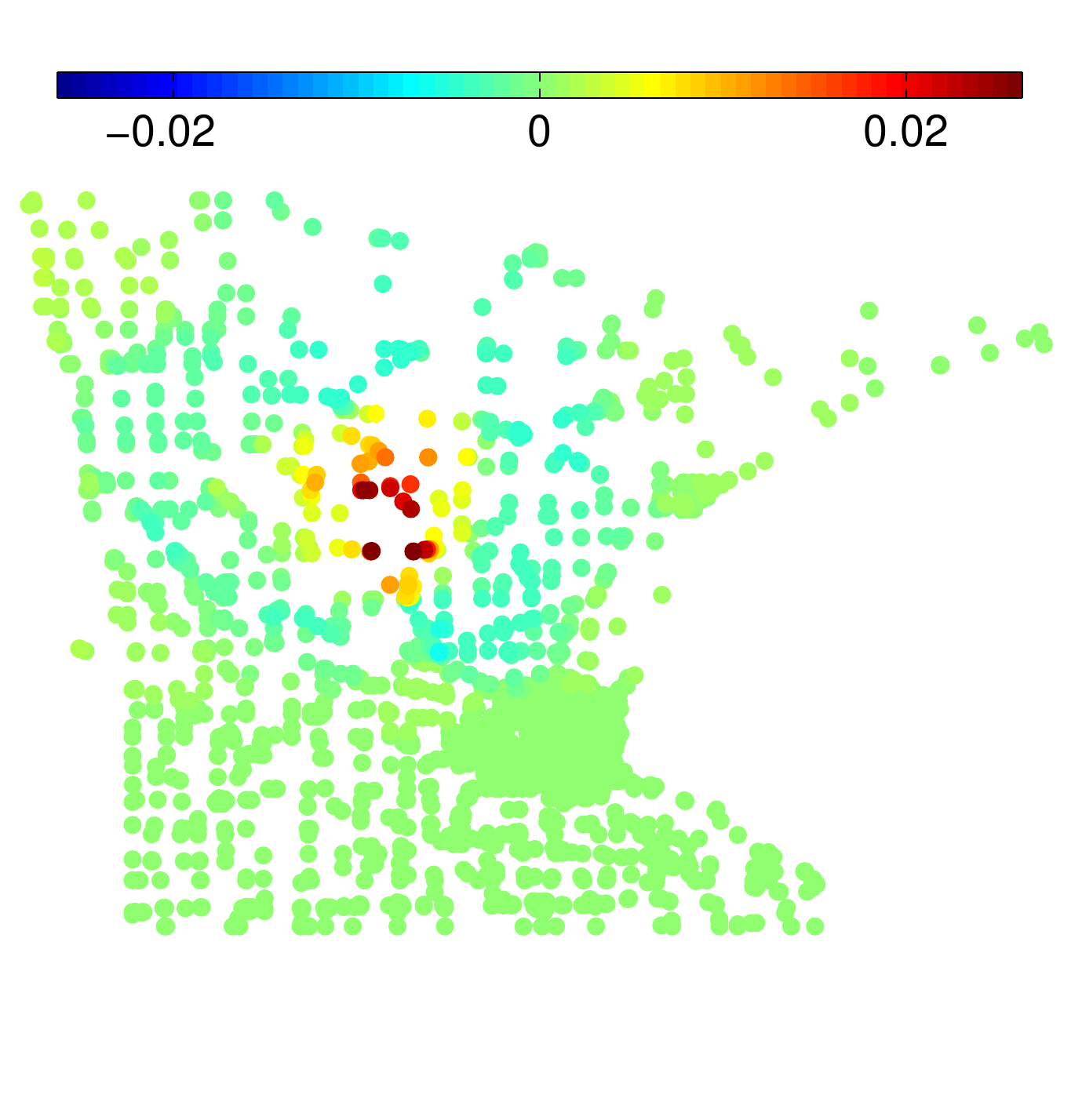} \\
(a) & (b) & (c) \\
\includegraphics[width=.32\columnwidth]{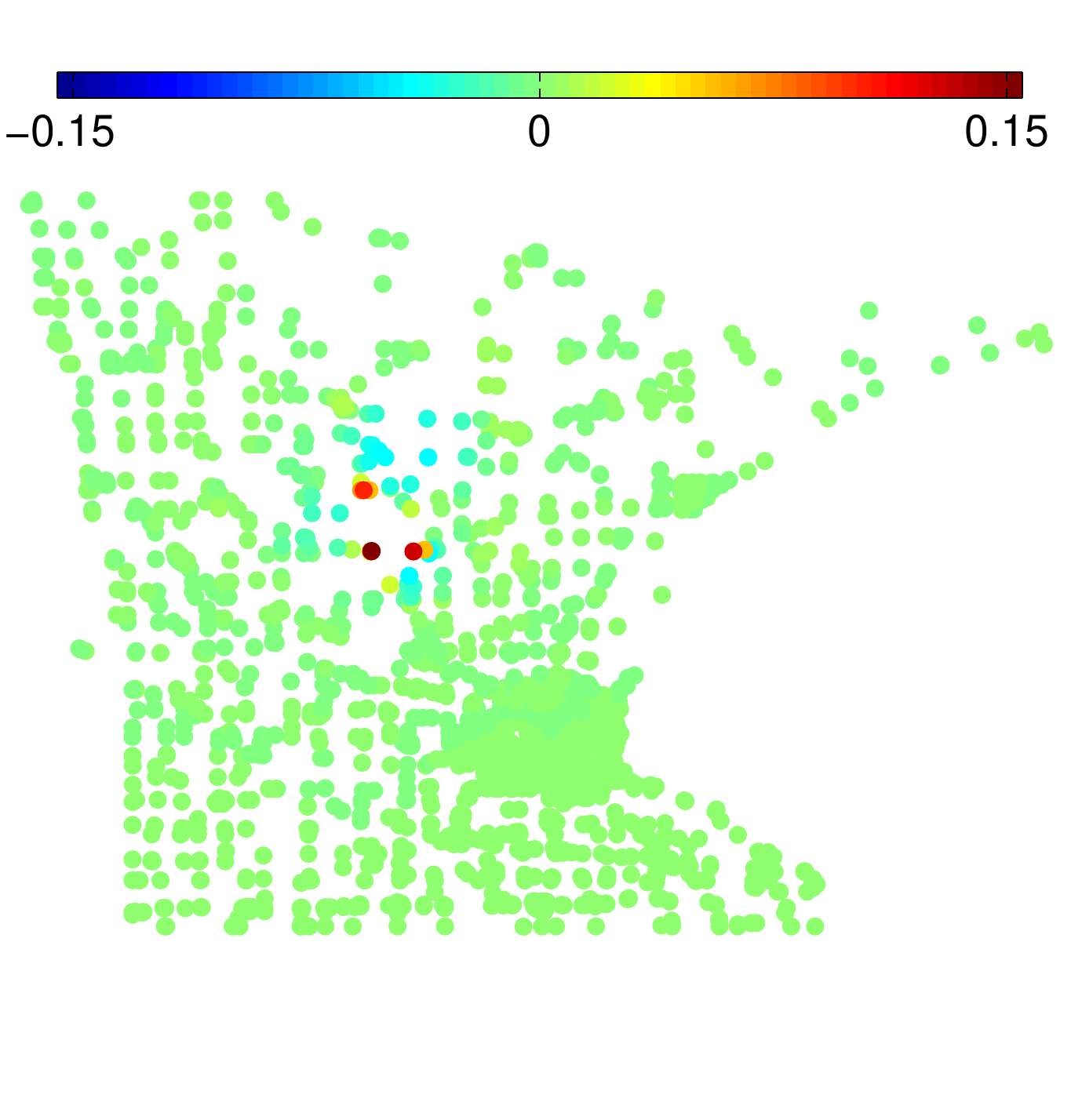} & 
\includegraphics[width=.32\columnwidth]{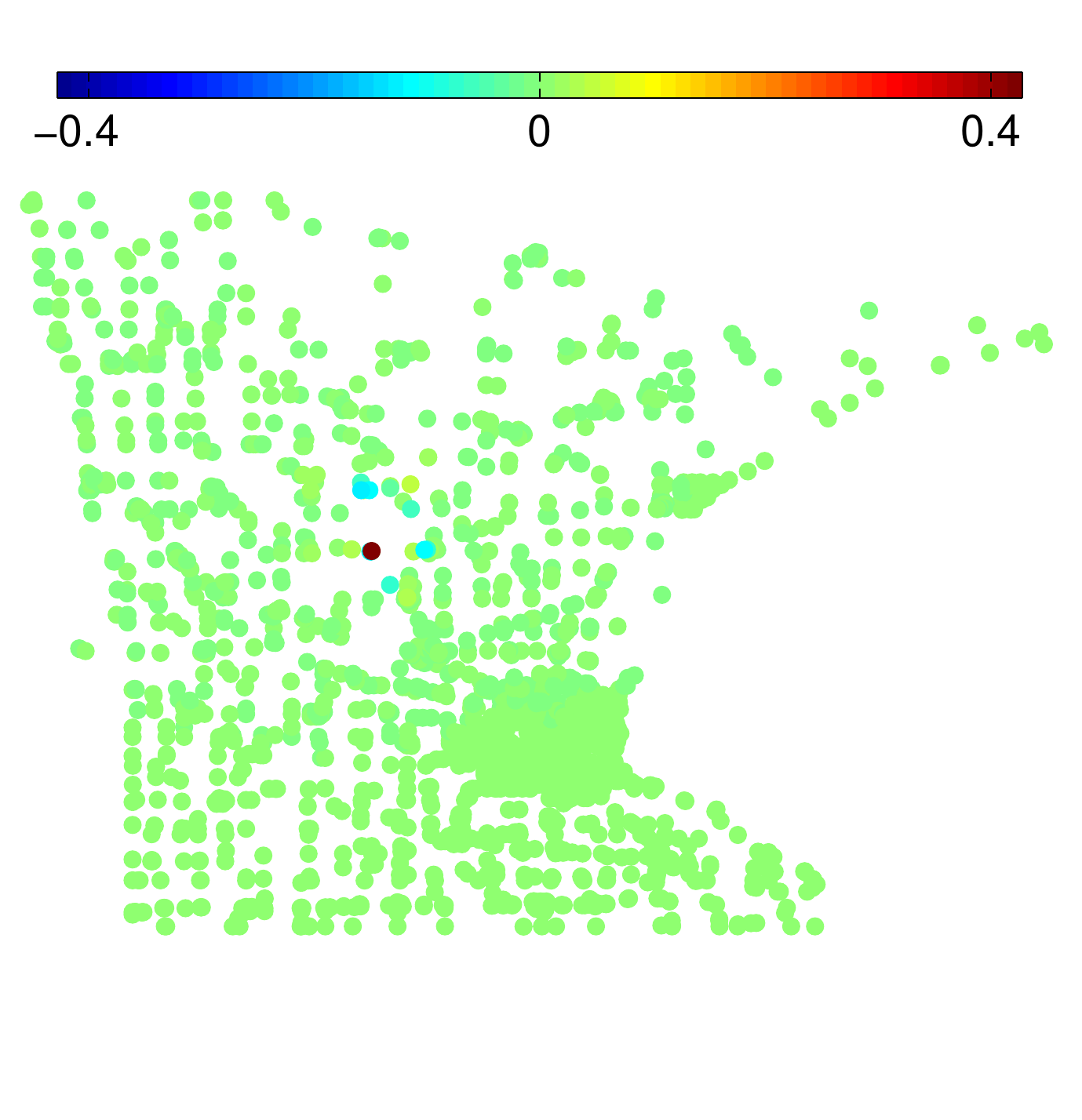} &
\includegraphics[width=.32\columnwidth]{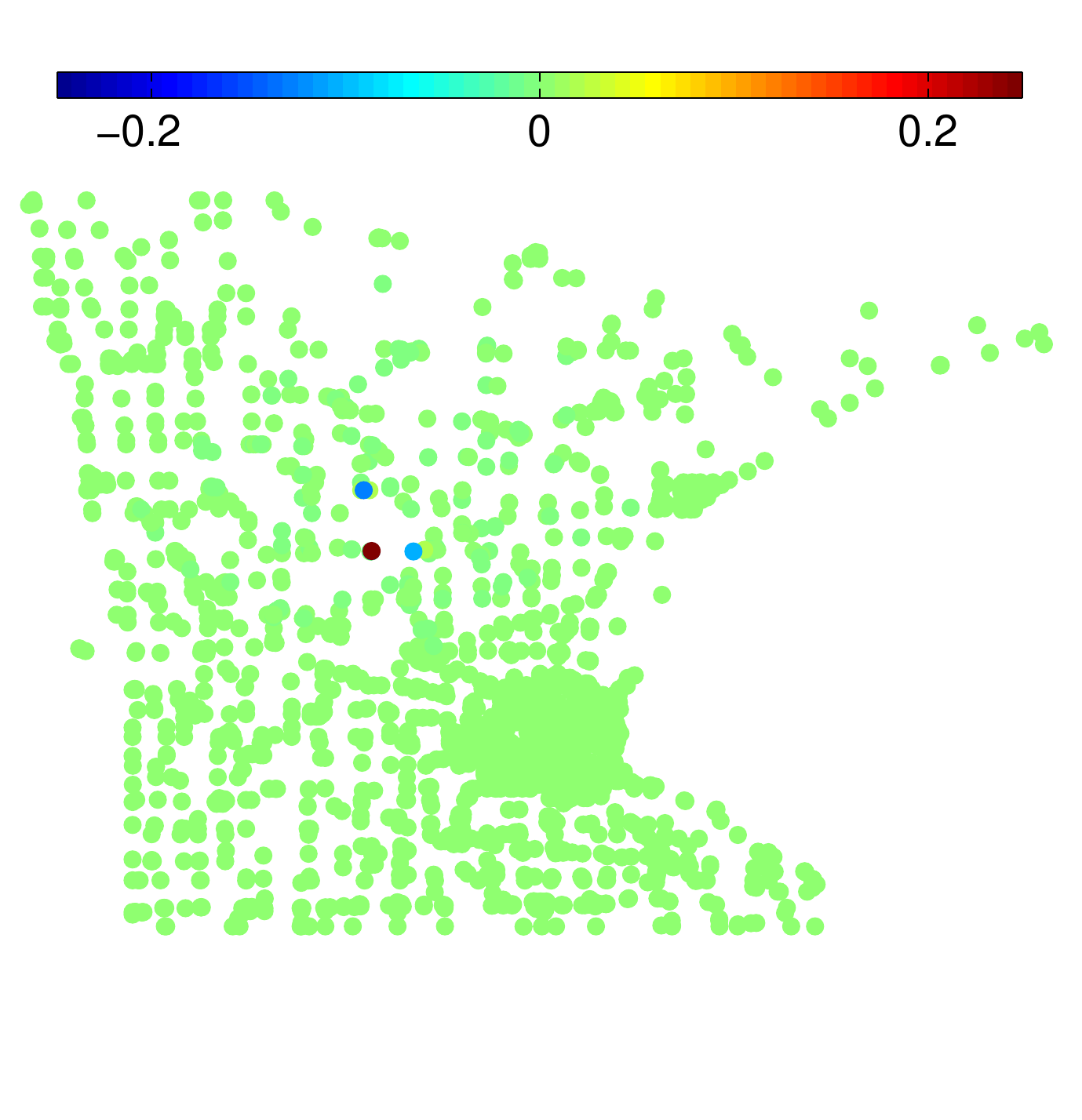} \\
(d) & (e) & (f) \\
\end{tabular}
\caption{Spectral graph wavelets on Minnesota road graph, with
  $K=100$, $J=4$ scales. (a) vertex at which wavelets are centered (b)
  scaling function (c)-(f) wavelets, scales 1-4.} \label{fig:minnesota}
\end{figure}

A second example is provided by a transportation network. In Figure
\ref{fig:minnesota} we consider a graph describing the road network
for Minnesota. In this dataset, edges represent major roads and
vertices their intersection points, which often but not always
correspond to towns or cities. For this example the graph is
unweighted, i.e. the edge weights are all equal to unity independent
of the physical length of the road segment represented. In particular,
the spatial coordinates of each vertex are used only for displaying
the graph and the corresponding wavelets, but do not affect the edge
weights. We show wavelets constructed with $K=100$ and $J=4$ scales.

Graph wavelets on transportation networks could prove useful for
analyzing data measured at geographical locations where one would
expect the underlying phenomena to be influenced by movement of people
or goods along the transportation infrastructure. Possible example
applications of this type include analysis of epidemiological data
describing the spread of disease, analysis of inventory of trade goods
(e.g. gasoline or grain stocks) relevant for logistics problems, or
analysis of census data describing human migration patterns.

\begin{figure}[t]
\begin{tabular}{c@{}c@{}c}
\includegraphics[width=.32\columnwidth]{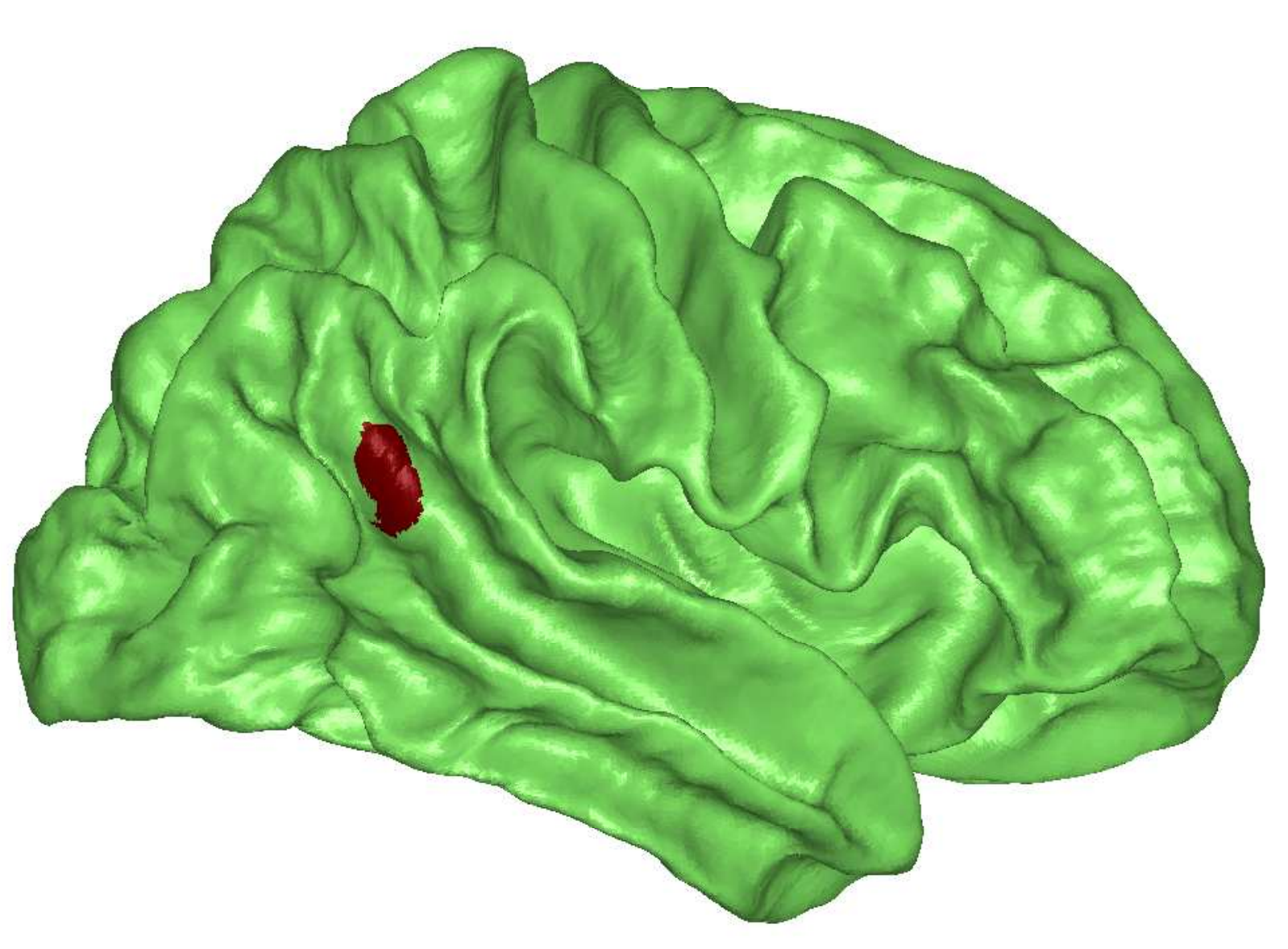} & 
\includegraphics[width=.32\columnwidth]{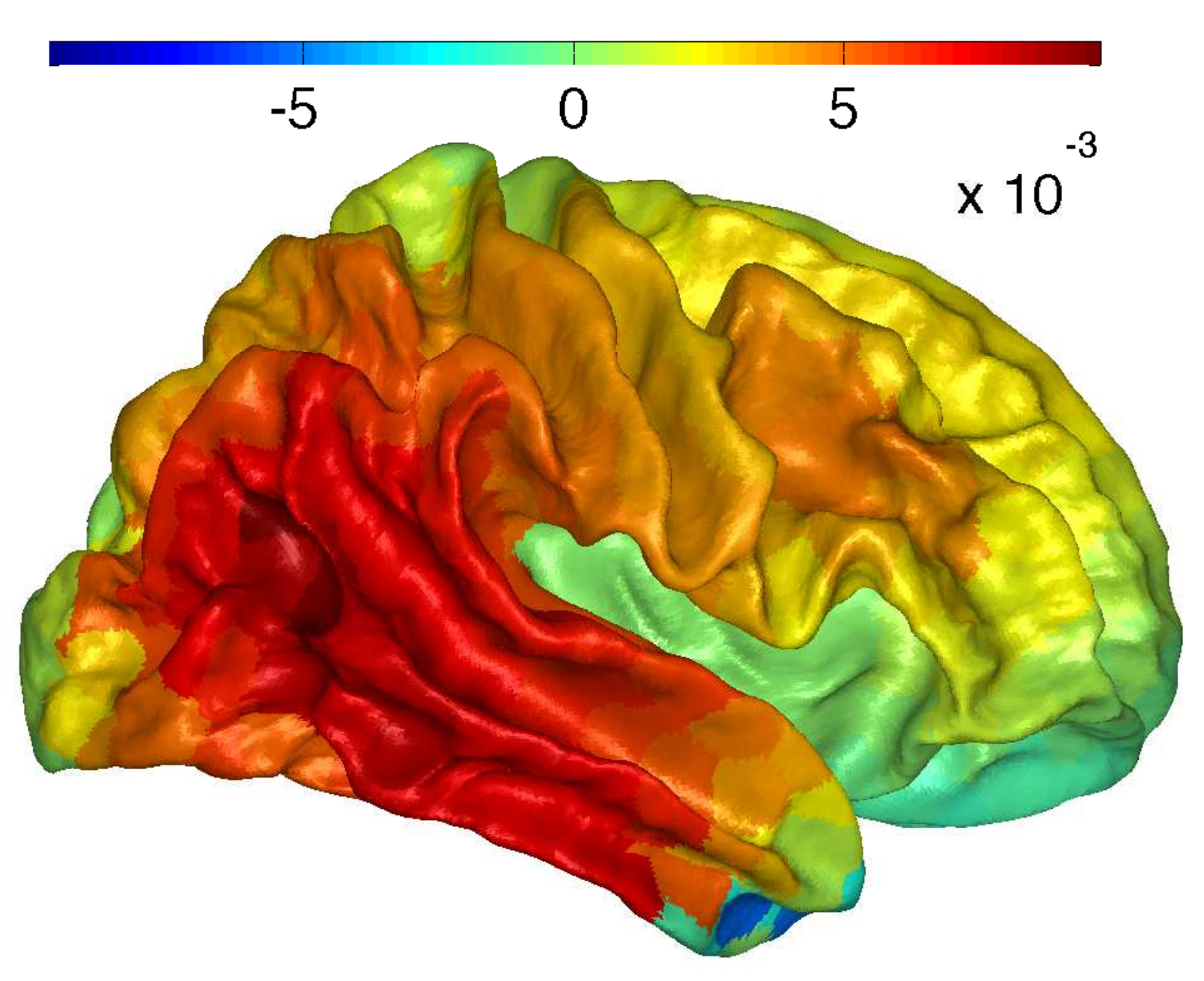} &
\includegraphics[width=.32\columnwidth]{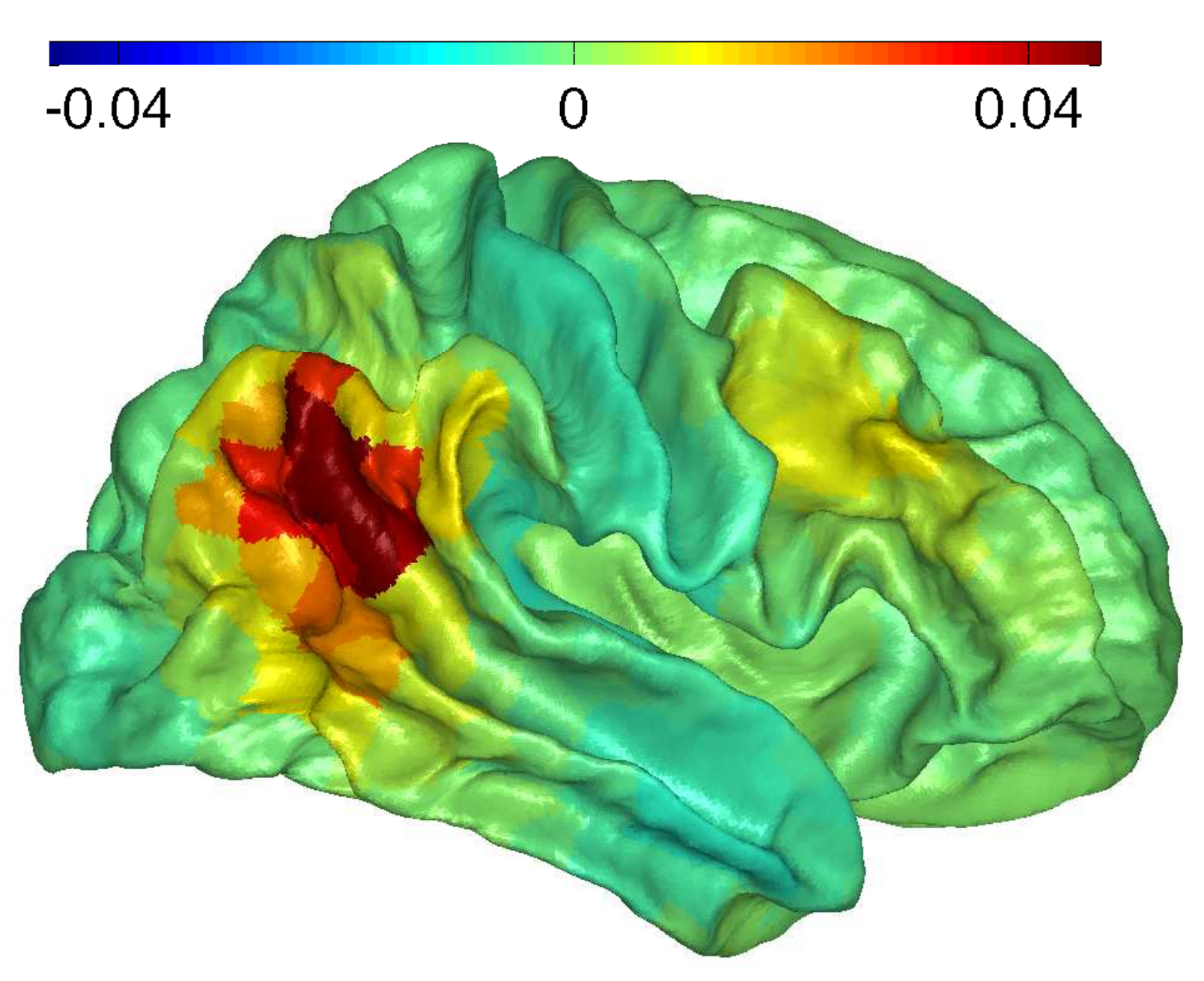} \\
(a) & (b) & (c) \\
\includegraphics[width=.32\columnwidth]{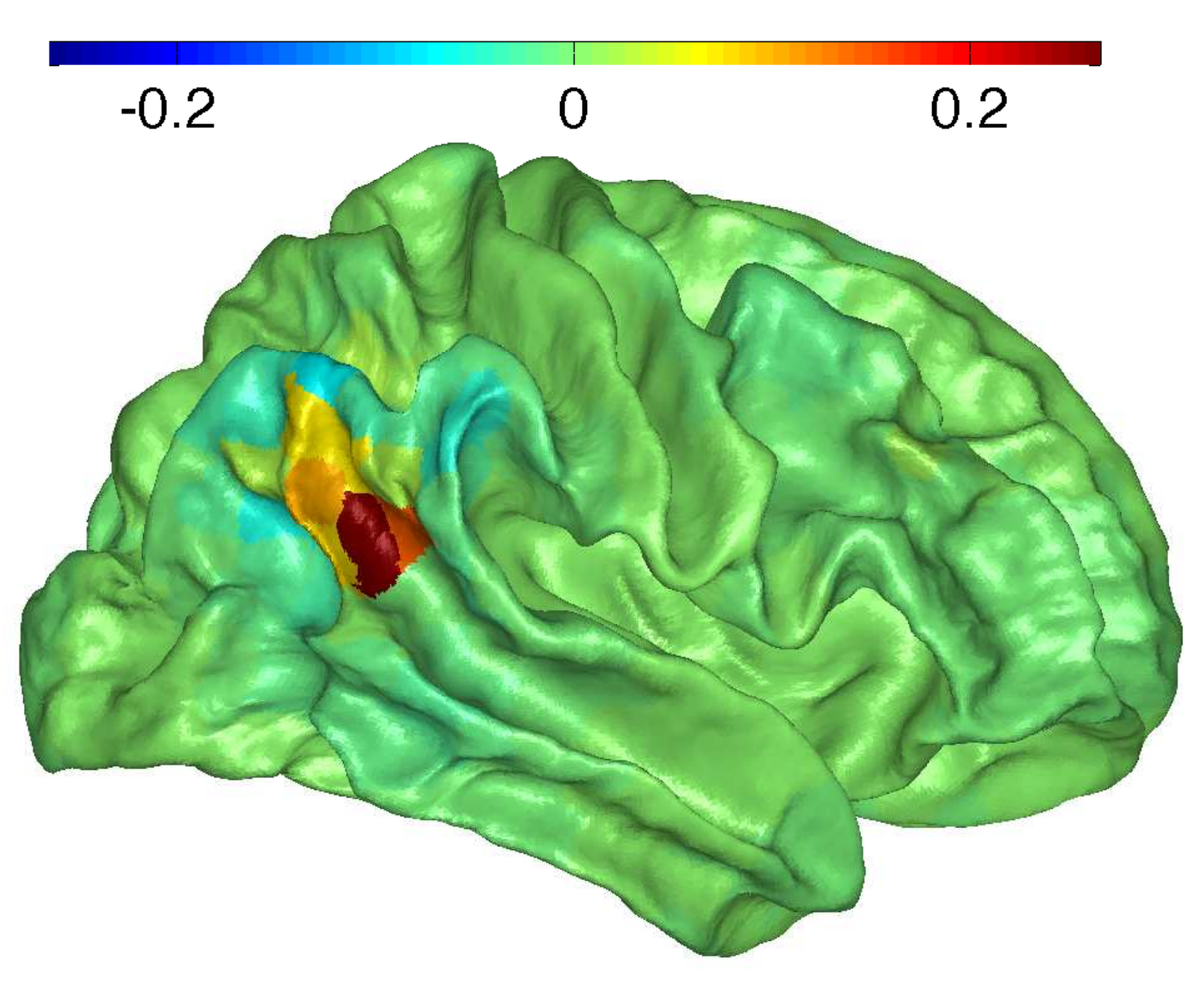} & 
\includegraphics[width=.32\columnwidth]{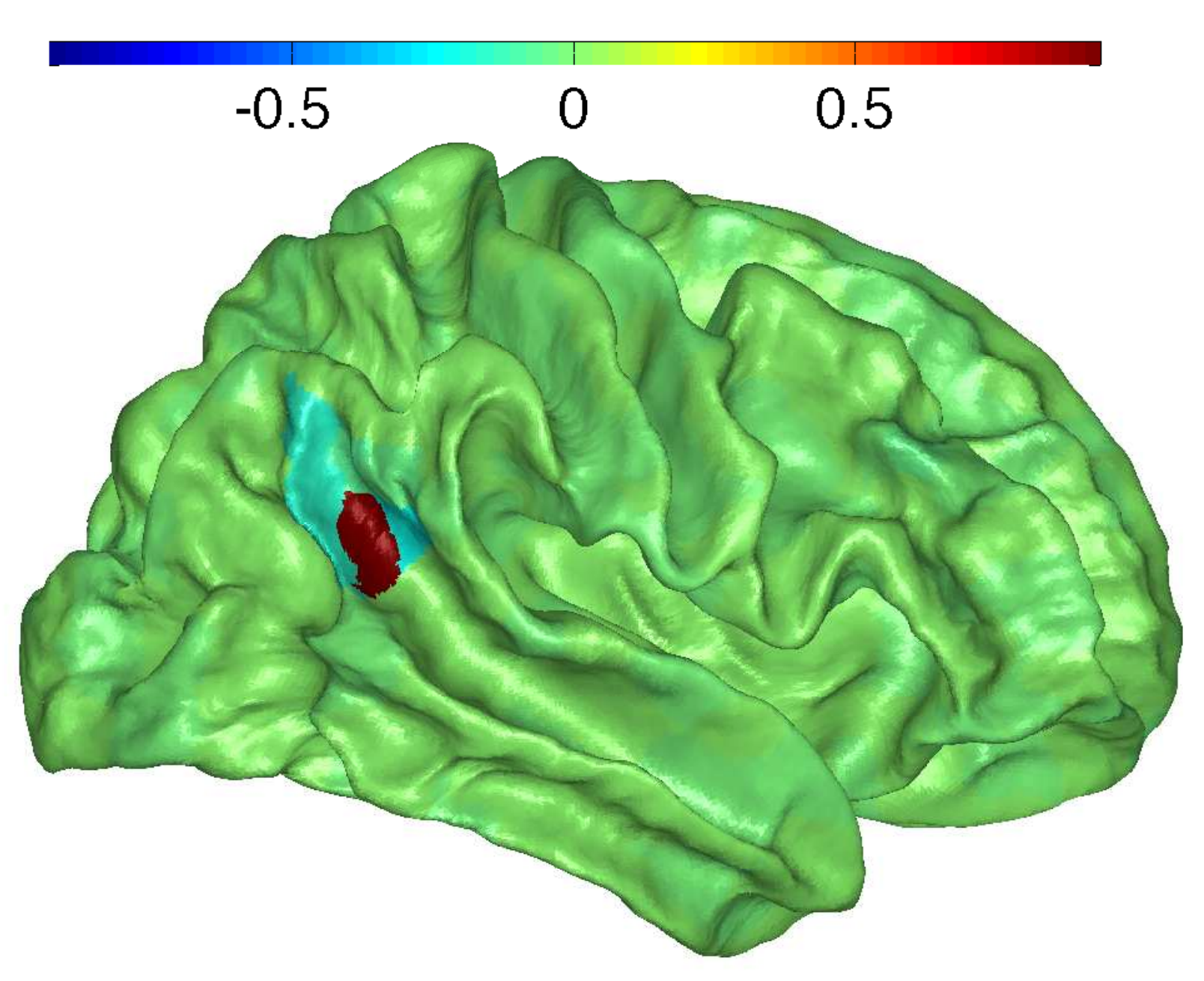} &
\includegraphics[width=.32\columnwidth]{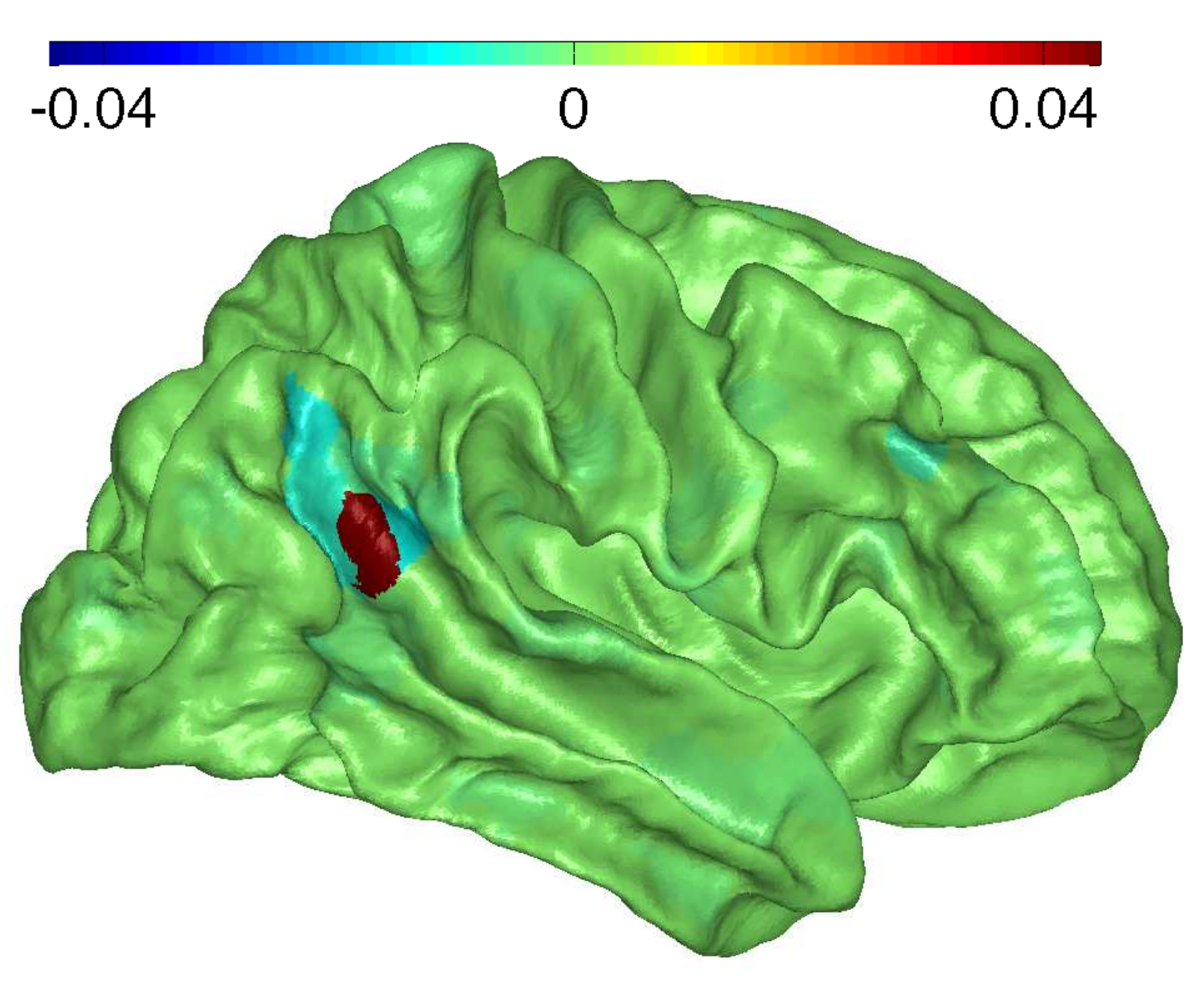} \\
(d) & (e) & (f) \\
\end{tabular}
\caption{Spectral graph wavelets on cerebral cortex, with $K=50$,
  $J=4$ scales. (a) ROI at which wavelets are centered (b) scaling
  function (c)-(f) wavelets, scales 1-4.} \label{fig:cortical}
\end{figure}

Another promising potential application of the spectral graph wavelet
transform is for use in data analysis for brain imaging. Many brain
imaging modalities, notably functional MRI, produce static or time
series maps of activity on the cortical surface. Functional MRI
imaging attempts to measure the difference between ``resting'' and
``active'' cortical states, typically by measuring MRI signal
correlated with changes in cortical blood flow. Due to both
constraints on imaging time and the very indirect nature of the
measurement, functional MRI images typically have a low
signal-to-noise ratio. There is thus a need for techniques for dealing
with high levels of noise in functional MRI images, either through
direct denoising in the image domain or at the level of statistical
hypothesis testing for defining active regions.

Classical wavelet methods have been studied for use in fMRI
processing, both for denoising in the image domain \cite{Zaroubi2000}
and for constructing statistical hypothesis testing
\cite{Ruttimann1998, Ville2004}. The power of these methods relies on
the assumption that the underlying cortical activity signal is
spatially localized, and thus can be efficiently represented with
localized wavelet waveforms. However, such use of wavelets ignores the
anatomical connectivity of the cortex.


A common view of the cerebral cortex is that it is organized into
distinct functional regions which are interconnected by tracts of
axonal fibers. Recent advances in diffusion MRI imaging, notable
diffusion tensor imaging (DTI) and diffusion spectrum imaging (DSI),
have enabled measuring the directionality of fiber tracts in the
brain. By tracing the fiber tracts, it is possible to non-invasively
infer the anatomical connectivity of cortical regions. This raises an
interesting question of whether knowledge of anatomical connectivity
can be exploited for processing of image data on the cortical surface.

We \footnote{In collaboration with Dr Leila Cammoun and
  Prof. Jean-Philippe Thiran, EPFL, Lausanne, Dr Patric Hagmann and
  Prof. Reto Meuli, CHUV, Lausanne} have begun to address this issue
by implementing the spectral graph wavelets on a weighted graph which
captures the connectivity of the cortex. Details of measuring the
cortical connection matrix are described in \cite{Hagmann2008}. Very
briefly, the cortical surface is first subdivided into 998 Regions of
Interest (ROI's). A large number of fiber tracts are traced, then the
connectivity of each pair of ROI's is proportional to the number of
fiber tracts connecting them, with a correction term depending on the
measured fiber length. The resulting symmetric matrix can be viewed as
a weighted graph where the vertices are the ROI's. Figure
\ref{fig:cortical} shows example spectral graph wavelets computed on
the cortical connection graph, visualized by mapping the ROI's back
onto a 3d model of the cortex. Only the right hemisphere is shown,
although the wavelets are defined on both hemispheres. For future work
we plan to investigate the use of these cortical graph wavelets for
use in regularization and denoising of functional MRI data.

\begin{figure}[t]
\begin{tabular}{cc}
\includegraphics[width=.47\columnwidth]{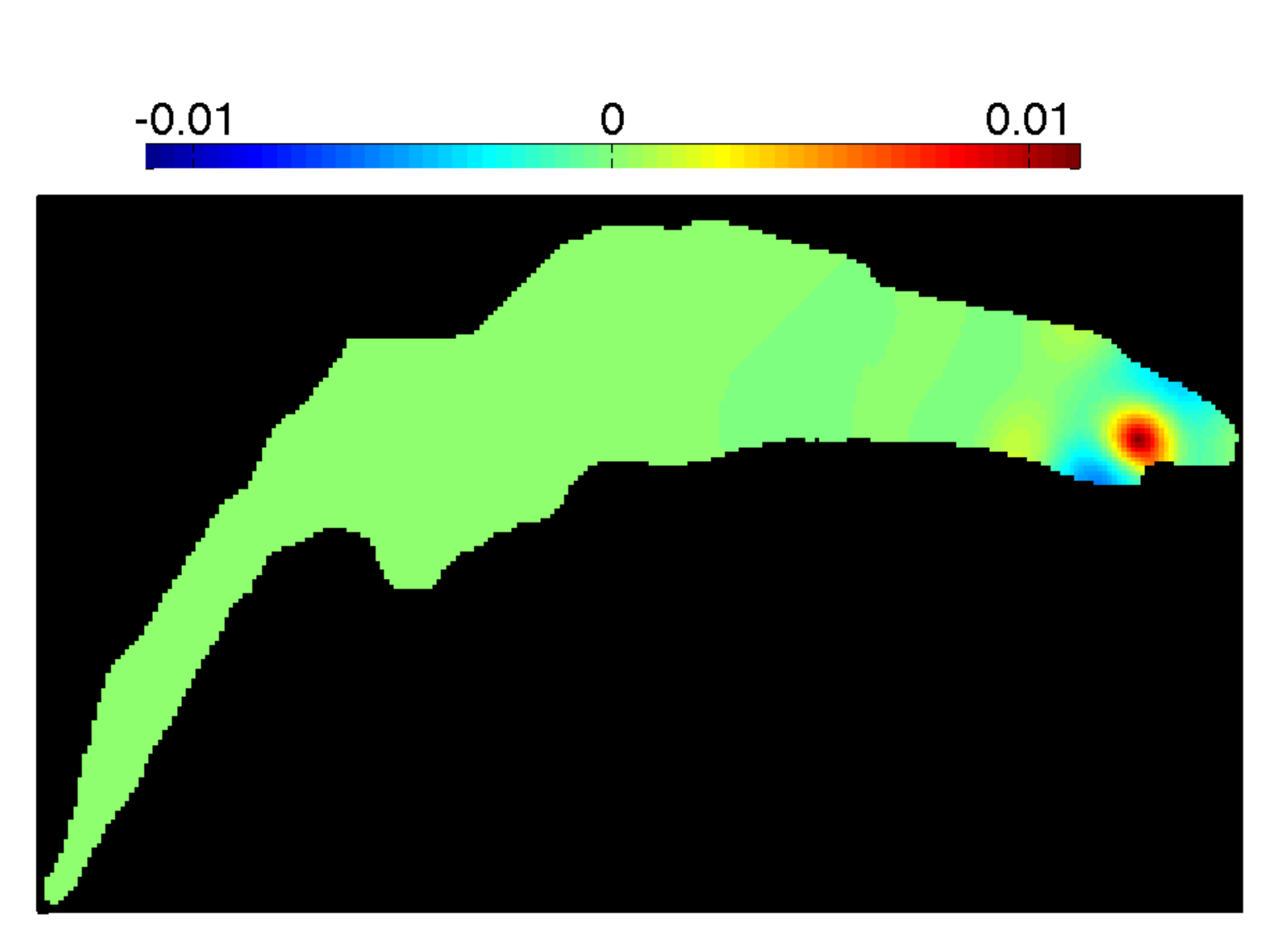} & 
\includegraphics[width=.47\columnwidth]{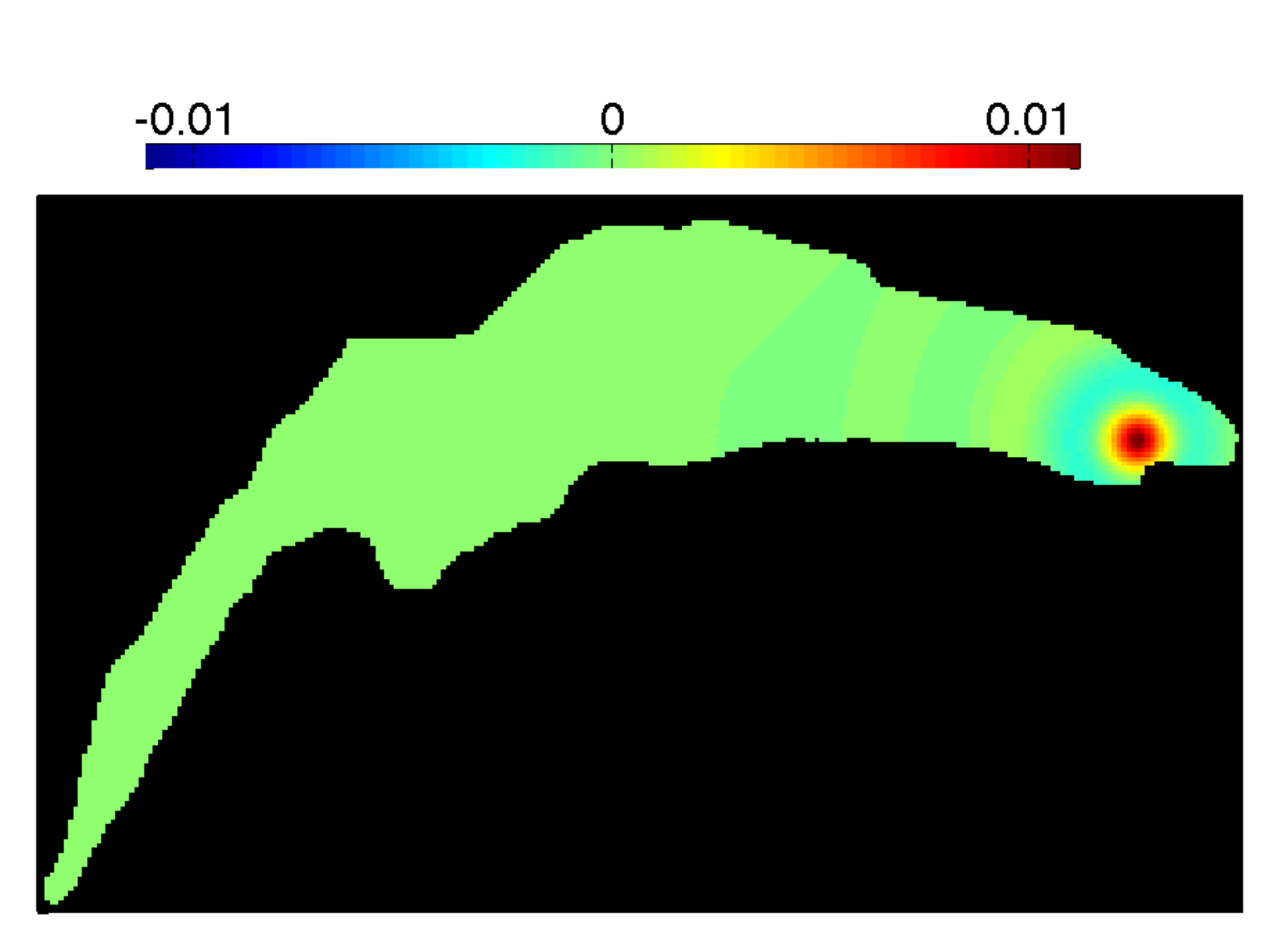} \\
(a) & (b) \\
\includegraphics[width=.47\columnwidth]{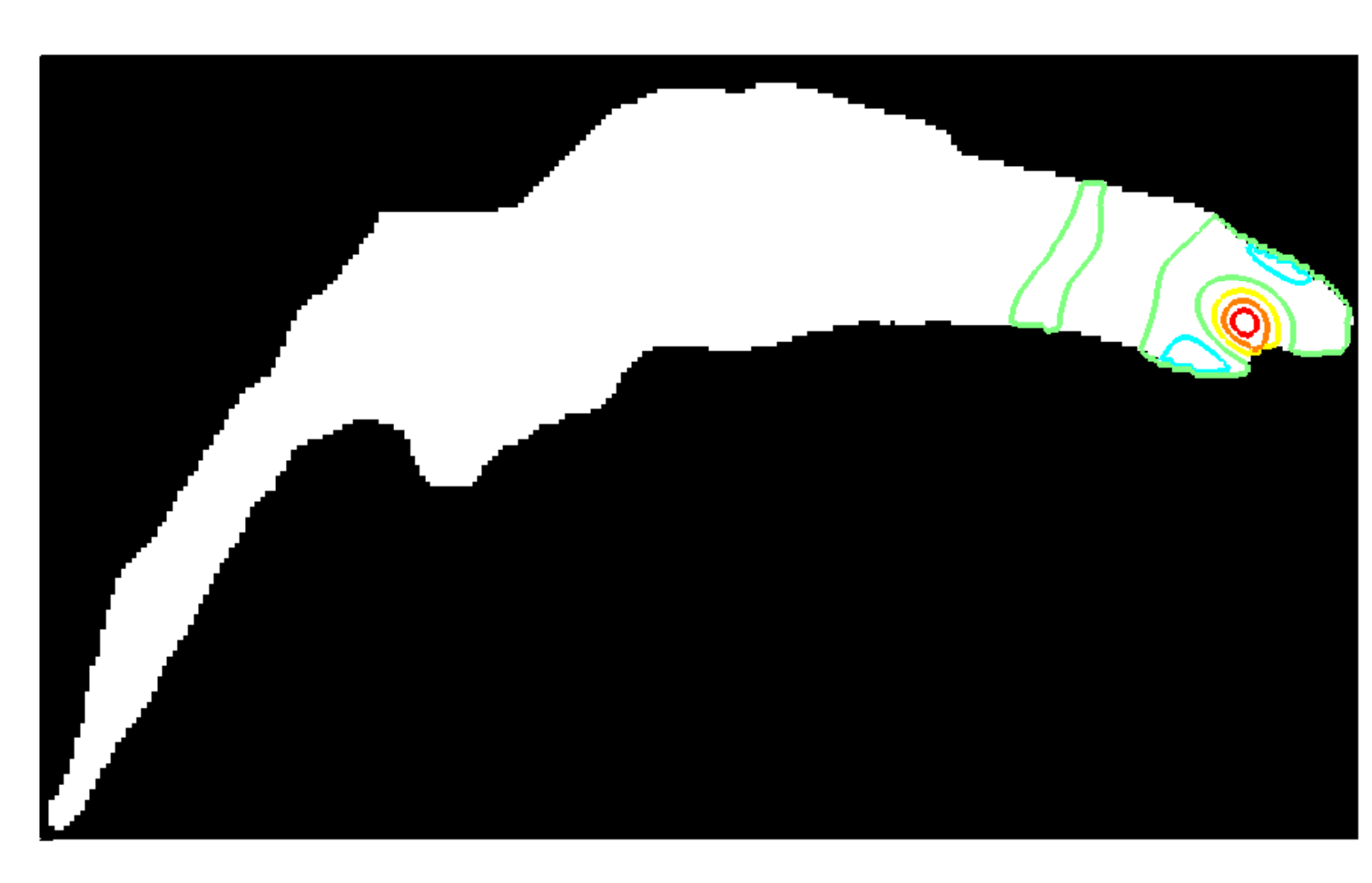} & 
\includegraphics[width=.47\columnwidth]{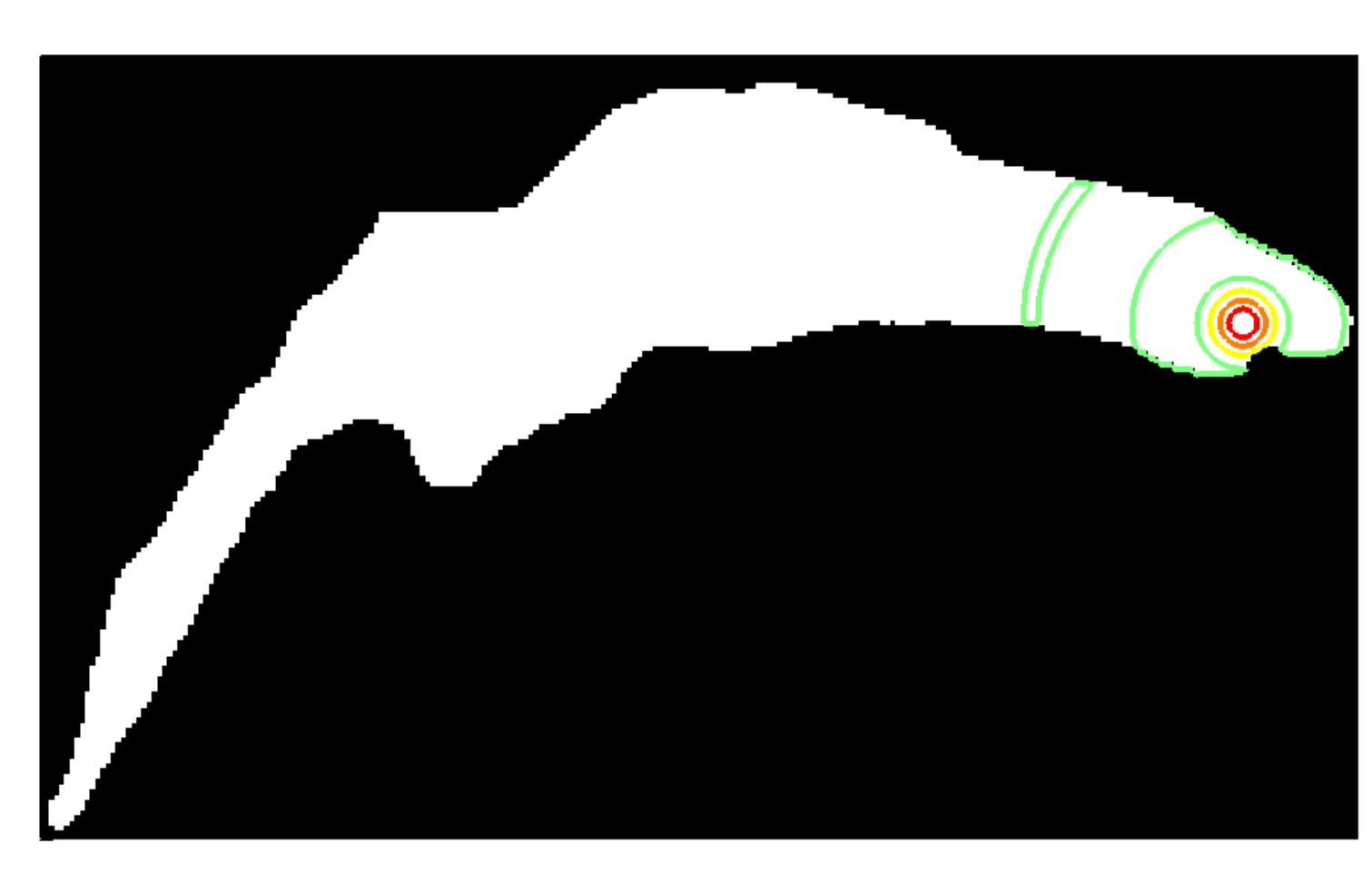} \\
(c) & (d) \\
\end{tabular}
\caption{Spectral graph wavelets on lake Geneva domain, (spatial map
  (a), contour plot (c)); compared with truncated wavelets from graph
  corresponding to complete mesh (spatial map (b), contour plot
  (d)). Note that the graph wavelets adapt to the geometry of the
  domain. } \label{fig:lakelets}
\end{figure}

A final interesting application for the spectral graph wavelet
transform is the construction of wavelets on irregularly shaped
domains.  As a representative example, consider that for some problems
in physical oceanography one may need to manipulate scalar data, such
as water temperature or salinity, that is only defined on the surface
of a given body of water. In order to apply wavelet analysis for such
data, one must adapt the transform to the potentially very complicated
boundary between land and water.  The spectral wavelets handle the
boundary implicitly and gracefully.  As an illustration we examine the
spectral graph wavelets where the domain is determined by the surface
of a lake.

For this example the lake domain is given as a mask defined on a
regular grid. We construct the corresponding weighted graph having
vertices that are grid points inside the lake, and retaining only
edges connecting neighboring grid points inside the lake. We set all
edge weights to unity. The corresponding graph Laplacian is thus
exactly the 5-point stencil (\ref{eq:laplacian_stencil}) for
approximating the continuous operator $-\nabla^2$ on the interior of
the domain; while at boundary points the graph Laplacian is modified
by the deletion of edges leaving the domain. We show an example
wavelet on Lake Geneva in Figure \ref{fig:lakelets}. Shoreline data
was taken from the GSHHS database \cite{Wessel1996} and the lake mask
was created on a 256 x 153 pixel grid using an azimuthal equidistant
projection, with a scale of 232 meters/pixel. The wavelet displayed is
from the coarsest wavelet scale, using the generating kernel described
in \ref{sec:sgwt_design} with parameters $K=100$ and $J=5$ scales.

For this type of domain derived by masking a regular grid, one may
compare the wavelets with those obtained by simply truncating the
wavelets derived from a large regular grid. As the wavelets have
compact support, the true and truncated wavelets will coincide for
wavelets located far from the irregular boundary. As can be seen in
Figure \ref{fig:lakelets}, however, they are quite different for
wavelets located near the irregular boundary. This comparison gives
direct evidence for the ability of the spectral graph wavelets to
adapt gracefully and automatically to the arbitrarily shaped domain.

We remark that the regular sampling of data within the domain may be
unrealistic for problems where data are collected at irregularly
placed sensor locations. The spectral graph wavelet transform could
also be used in this case by constructing a graph with vertices at the
sensor locations, however we have not considered such an example here.

\section{Conclusions and Future Work}
\label{sec:conclusion}

We have presented a framework for constructing wavelets on arbitrary
weighted graphs. By analogy with classical wavelet operators in the
Fourier domain, we have shown that scaling may be implemented in the
spectral domain of the graph Laplacian. We have shown that the
resulting spectral graph wavelets are localized in the small scale
limit, and form a frame with easily calculable frame bounds. We have
detailed an algorithm for computing the wavelets based on Chebyshev
polynomial approximation that avoids the need for explicit
diagonalization of the graph Laplacian, and allows the application of
the transform to large graphs. Finally we have shown examples of the
wavelets on graphs arising from several different potential
application domains.

There are many possible directions for future research for improving
or extending the SGWT. One property of the transform presented here is
that, unlike classical orthogonal wavelet transforms, we do not
subsample the transform at coarser spatial scales. As a result the
SGWT is overcomplete by a factor of J+1 where J is the number of
wavelet scales. Subsampling of the SGWT can be determined by selecting
a mask of vertices at each scale corresponding to the centers of the
wavelets to preserve. This is a more difficult problem on an arbitrary
weighted graph than on a regular mesh, where one may exploit the
regular geometry of the mesh to perform dyadic subsampling at each
scale. An interesting question for future research would be to
investigate an appropriate criterion for determining a good selection
of wavelets to preserve after subsampling. As an example, one may
consider preserving the frame bounds as much as possible under the
constraint that the overall overcompleteness should not exceed a
specified factor.

A related question is to consider how the SGWT would interact with
graph contraction. A weighted graph may be contracted by partitioning
its vertices into disjoint sets; the resulting contracted graph has
vertices equal to the number of partitions and edge weights determined
by summing the weights of the edges connecting any two
partitions. Repeatedly contracting a given weighted graph could define
a multiscale representation of the weighted graph. Calculating a
single scale of the spectral graph wavelet transform for each of these
contracted graphs would then yield a multiscale wavelet analysis.
This proposed scheme is inspired conceptually by the fast wavelet
transform for classical orthogonal wavelets, based on recursive
filtering and subsampling. The question of how to automatically define
the contraction at each scale on an arbitrary irregular graph is
itself a difficult research problem.

The spectral graph wavelets presented here are not directional.  In
particular when constructed on regular meshes they yield radially
symmetric waveforms. This can be understood as in this case the graph
Laplacian is the discretization of the isotropic continuous Laplacian.
In the field of image processing, however, it has long been recognized
that directionally selective filters are more efficient at
representing image structure.  This raises the interesting question of
how, and when, graph wavelets can be constructed which have some
directionality. Intuitively, this will require some notion of local
directionality, i.e. some way of defining directions of all of the
neighbors of a given vertex. As this would require the definition of
additional structure beyond the raw connectivity information, it may
not be appropriate for completely arbitrary graphs.  For graphs which
arise from sampling a known orientable manifold, such as the meshes
with irregular boundary used in Figure \ref{fig:lakelets}, one may
infer such local directionality from the original manifold.

For some problems it may be useful to construct graphs that mix both
local and non-local connectivity information. As a concrete example
consider the cortical graph wavelets shown in Figure
\ref{fig:cortical}. As the vertices of the graph correspond to sets of
MRI voxels grouped into ROI's, the wavelets are defined on the ROI's
and thus cannot be used to analyze data defined on the scale of
individual voxels.  Analyzing voxel scale data with the SGWT would
require constructing a graph with vertices corresponding to individual
voxels. However, the nonlocal connectivity is defined only on the
scale of the ROI's.  One way of defining the connectivity for the
finer graph would be as a sum $A^{nonlocal} + A^{local}$, where
$A^{nonlocal}_{m,n}$ is the weight of the connection between the ROI
containing vertex $m$ and the ROI containing vertex $n$, and
$A^{local}_{m,n}$ indexes whether $m$ and $n$ are spatial
neighbors. Under this scheme we consider $A^{local}$ as implementing a
``default'' local connectivity not arising from any particular
measurement.  Considering this raises interesting questions of how to
balance the relative contributions of the local and nonlocal
connectivities, as well as how the special structure of the hybrid
connectivity matrix could be exploited for efficient computation.  

The particular form of the wavelet generating kernel $g$ used in the
examples illustrating this work was chosen in a somewhat ad-hoc
manner. Aside from localization in the small-scale limit which
required polynomial behaviour of $g$ at the origin, we have avoided
detailed analysis of how the choice of $g$ affects the wavelets. In
particular, we have not chosen $g$ and the choice of spatial scales to
optimize the resulting frame bounds. More detailed investigation is
called for regarding optimizing the design of $g$ for different
applications.

The fast Chebyshev polynomial approximation scheme we describe here
could itself be useful independent of its application for computing
the wavelet transform. One application could be for filtering of data
on irregularly shaped domains, such as described in Figure
\ref{fig:lakelets}. For example, smoothing data on such a domain by
convolving with a Gaussian kernel is confounded by the problem that
near the edges the kernel would extend off of the domain.  As an
alternative, one could express the convolution as multiplication by a
function in the Fourier domain, approximate this function with a
Chebyshev polynomial, and then apply the algorithm described in this
paper. 
This could also be used for band-pass or high-pass filtering of data
on irregular domains, by designing appropriate filters in the spectral
domain induced by the graph Laplacian.

The Chebyshev approximation scheme may also be useful for machine
learning problems on graphs. Some recent work has studied using the
``diffusion kernel'' $K_t=e^{-t\L}$ for use with kernel-based machine
learning algorithms \cite{Kondor2002}. The Chebyshev polynomial scheme
provides a fast way to approximate this exponential that may be useful
for large problems on unstructured yet sparse graphs.


\bibliography{hammond-vandergheynst-gribonval-acha-2009-arxiv}

\begin{thebibliography}{10}
\expandafter\ifx\csname url\endcsname\relax
  \def\url#1{\texttt{#1}}\fi
\expandafter\ifx\csname urlprefix\endcsname\relax\def\urlprefix{URL }\fi
\expandafter\ifx\csname href\endcsname\relax
  \def\href#1#2{#2} \def\path#1{#1}\fi

\bibitem{Shapiro1993}
J.~Shapiro, Embedded image coding using zerotrees of wavelet coefficients,
  Signal Processing, IEEE Transactions on 41~(12) (1993) 3445--3462.

\bibitem{Said1996}
A.~Said, W.~Pearlman, A new, fast, and efficient image codec based on set
  partitioning in hierarchical trees, Circuits and Systems for Video
  Technology, IEEE Transactions on 6~(3) (1996) 243--250.

\bibitem{Hilton1997}
M.~Hilton, Wavelet and wavelet packet compression of electrocardiograms,
  Biomedical Engineering, IEEE Transactions on 44~(5) (1997) 394--402.

\bibitem{Buccigrossi1999}
R.~Buccigrossi, E.~Simoncelli, Image compression via joint statistical
  characterization in the wavelet domain, Image Processing, IEEE Transactions
  on 8~(12) (1999) 1688--1701.

\bibitem{Taubman2002}
D.~Taubman, M.~Marcellin, {JPEG}2000 : Image compression fundamentals,
  standards and practice, Kluwer Academic Publishers, 2002.

\bibitem{Donoho1994}
D.~L. Donoho, I.~M. Johnstone, Ideal spatial adaptation by wavelet shrinkage,
  Biometrika 81 (1994) 425--455.

\bibitem{Chang2000}
S.~Chang, B.~Yu, M.~Vetterli, Adaptive wavelet thresholding for image denoising
  and compression, Image Processing, IEEE Transactions on 9~(9) (2000)
  1532--1546.

\bibitem{Sendur2002}
L.~Sendur, I.~Selesnick, Bivariate shrinkage functions for wavelet-based
  denoising exploiting interscale dependency, Signal Processing, IEEE
  Transactions on 50~(11) (2002) 2744--2756.

\bibitem{Portilla2003}
J.~Portilla, V.~Strela, M.~J. Wainwright, E.~P. Simoncelli, Image denoising
  using scale mixtures of {G}aussians in the wavelet domain, IEEE Transactions
  on Image Processing 12 (2003) 1338--1351.

\bibitem{Daubechies2005}
I.~Daubechies, G.~Teschke, Variational image restoration by means of wavelets:
  Simultaneous decomposition, deblurring, and denoising, Applied and
  Computational Harmonic Analysis 19~(1) (2005) 1 -- 16.

\bibitem{Starck1994}
J.-L. Starck, A.~Bijaoui, Filtering and deconvolution by the wavelet transform,
  Signal Processing 35~(3) (1994) 195 -- 211.

\bibitem{Donoho1995}
D.~L. Donoho, Nonlinear solution of linear inverse problems by
  wavelet-vaguelette decomposition, Applied and Computational Harmonic Analysis
  2~(2) (1995) 101 -- 126.

\bibitem{Miller1995}
E.~Miller, A.~S. Willsky, A multiscale approach to sensor fusion and the
  solution of linear inverse problems, Applied and Computational Harmonic
  Analysis 2~(2) (1995) 127 -- 147.

\bibitem{Nowak2000}
R.~Nowak, E.~Kolaczyk, A statistical multiscale framework for {P}oisson inverse
  problems, Information Theory, IEEE Transactions on 46~(5) (2000) 1811--1825.

\bibitem{Bioucas-Dias2006}
J.~Bioucas-Dias, Bayesian wavelet-based image deconvolution: a {GEM} algorithm
  exploiting a class of heavy-tailed priors, Image Processing, IEEE
  Transactions on 15~(4) (2006) 937--951.

\bibitem{Manthalkar2003}
R.~Manthalkar, P.~K. Biswas, B.~N. Chatterji, Rotation and scale invariant
  texture features using discrete wavelet packet transform, Pattern Recognition
  Letters 24~(14) (2003) 2455 -- 2462.

\bibitem{Flandrin1992}
P.~Flandrin, Wavelet analysis and synthesis of fractional {B}rownian motion,
  Information Theory, IEEE Transactions on 38~(2, Part 2) (1992) 910 -- 917.

\bibitem{Lowe1999}
D.~Lowe, Object recognition from local scale-invariant features, Computer
  Vision, 1999. The Proceedings of the Seventh IEEE International Conference on
  2 (1999) 1150 -- 1157 vol.2.

\bibitem{Apte1994}
C.~Apt\'{e}, F.~Damerau, S.~M. Weiss, Automated learning of decision rules for
  text categorization, ACM Trans. Inf. Syst. 12~(3) (1994) 233--251.

\bibitem{Chung1997}
F.~K. Chung, Spectral Graph Theory, Vol.~92 of CBMS Regional Conference Series
  in Mathematics, AMS Bookstore, 1997.

\bibitem{Mallat1998}
S.~Mallat, A Wavelet Tour of Signal Processing, Academic Press, 1998.

\bibitem{Burt1983}
P.~J. Burt, E.~H. Adelson, The {L}aplacian pyramid as a compact image code,
  IEEE Transactions on Communications 31~(4) (1983) 532--540.

\bibitem{Simoncelli1992}
E.~P. Simoncelli, W.~T. Freeman, E.~H. Adelson, D.~J. Heeger, Shiftable
  multi-scale transforms, IEEE Trans Information Theory 38~(2) (1992) 587--607,
  special Issue on Wavelets.

\bibitem{Kingsbury2001}
N.~Kingsbury, Complex wavelets for shift invariant analysis and filtering of
  signals, Applied and Computational Harmonic Analysis 10~(3) (2001) 234 --
  253.

\bibitem{Candes2003}
E.~Candes, D.~Donoho, New tight frames of curvelets and optimal representations
  of objects with piecewise {$C^2$} singularities, Communications on Pure and
  Applied Mathematics 57 (2003) 219--266.

\bibitem{Peyre2008obb}
G.~Peyr{\'e}, S.~Mallat, Orthogonal bandlet bases for geometric images
  approximation, Comm. Pure Appl. Math. 61~(9) (2008) 1173--1212.

\bibitem{Antoine1999}
J.~Antoine, P.~Vandergheynst, Wavelets on the 2-{S}phere: {A}
  {G}roup-{T}heoretical {A}pproach, Applied and {C}omputational {H}armonic
  {A}nalysis 7~(3) (1999) 262--291.

\bibitem{Wiaux2008}
Y.~Wiaux, J.~D. McEwen, P.~Vandergheynst, O.~Blanc, Exact reconstruction with
  directional wavelets on the sphere, Mon. Not. R. Astron. Soc. 388 (2008) 770.

\bibitem{Antoine2008}
J.-P. Antoine, I.~Bogdanova, P.~Vandergheynst, The continuous wavelet transform
  on conic sections, Int. J. Wavelet and {M}ultiresolution {I}nf. Process.
  6~(2) (2008) 137--156.

\bibitem{Crovella2003}
M.~Crovella, E.~Kolaczyk, Graph wavelets for spatial traffic analysis, INFOCOM
  2003. Twenty-Second Annual Joint Conference of the IEEE Computer and
  Communications Societies. IEEE 3 (2003) 1848 -- 1857 vol.3.

\bibitem{Smalter2009}
A.~Smalter, J.~Huan, G.~Lushington, Graph wavelet alignment kernels for drug
  virtual screening, Journal of Bioinformatics and Computational Biology 7
  (2009) 473--497.

\bibitem{Coifman2006}
R.~R. Coifman, M.~Maggioni, Diffusion wavelets, Applied and Computational
  Harmonic Analysis 21 (2006) 53--94.

\bibitem{Geller2009}
D.~Geller, A.~Mayeli, Continuous wavelets on compact manifolds, Mathematische
  Zeitschrift 262 (2009) 895--927.

\bibitem{Grossmann1984}
A.~Grossmann, J.~Morlet, Decomposition of {H}ardy functions into square
  integrable wavelets of constant shape, SIAM Journal on Mathematical Analysis
  15~(4) (1984) 723--736.

\bibitem{Hein2005}
M.~Hein, J.~Audibert, U.~von Luxburg, From graphs to manifolds - weak and
  strong pointwise consistency of graph {L}aplacians, in: P.~Auer, R.~Meir
  (Eds.), Proc. 18th Conf Learning Theory (COLT), Vol. 3559 of Lecture Notes in
  Computer Science, Springer-Verlag, 2005, pp. 470--485.

\bibitem{Singer2006}
A.~Singer, From graph to manifold {L}aplacian: The convergence rate, Applied
  and Computational Harmonic Analysis 21~(1) (2006) 128 -- 134, diffusion Maps
  and Wavelets.

\bibitem{Belkin2008}
M.~Belkin, P.~Niyogi, Towards a theoretical foundation for {L}aplacian-based
  manifold methods, Journal of Computer and System Sciences 74~(8) (2008) 1289
  -- 1308, learning Theory 2005.

\bibitem{Reed1980}
M.~Reed, B.~Simon, Methods of Modern Mathematical Physics Volume 1 : Functional
  Analysis, Academic Press, 1980.

\bibitem{Bondy2008}
J.~A. Bondy, U.~S.~R. Murty, Graph Theory, no. 244 in Graduate Texts in
  Mathematics, Springer, 2008.

\bibitem{Daubechies1992}
I.~Daubechies, Ten Lectures on Wavelets, Society for Industrial and Applied
  Mathematics, 1992.

\bibitem{Heil1989}
C.~E. Heil, D.~F. Walnut, Continuous and discrete wavelet transforms, SIAM
  Review 31~(4) (1989) 628--666.

\bibitem{Watkins2007}
D.~Watkins, The Matrix Eigenvalue Problem - {GR} and {K}rylov subspace methods,
  Society for Industrial and Applied Mathematics, 2007.

\bibitem{Sleijpen1996}
G.~L.~G. Sleijpen, H.~A.~V. der Vorst, A {J}acobi--{D}avidson iteration method
  for linear eigenvalue problems, SIAM Journal on Matrix Analysis and
  Applications 17~(2) (1996) 401--425.

\bibitem{Cheney1966}
E.~Cheney, Introduction to approximation theory, McGraw-Hill, New York, 1966.

\bibitem{Fraser1965}
W.~Fraser, A survey of methods of computing minimax and near-minimax polynomial
  approximations for functions of a single independent variable, J. Assoc.
  Comput. Mach. 12 (1965) 295--314.

\bibitem{Geddes1978}
K.~O. Geddes, Near-minimax polynomial approximation in an elliptical region,
  SIAM Journal on Numerical Analysis 15~(6) (1978) 1225--1233.

\bibitem{Phillips2003}
G.~M. Phillips, Interpolation and Approximation by Polynomials, CMS Books in
  Mathematics, Springer-Verlag, 2003.

\bibitem{Grochenig1993}
K.~Grochenig, Acceleration of the frame algorithm, Signal Processing, IEEE
  Transactions on 41~(12) (1993) 3331--3340.

\bibitem{Zaroubi2000}
S.~Zaroubi, G.~Goelman, Complex denoising of {MR} data via wavelet analysis:
  Application for functional {MRI}, Magnetic Resonance Imaging 18~(1) (2000) 59
  -- 68.

\bibitem{Ruttimann1998}
U.~Ruttimann, M.~Unser, R.~Rawlings, D.~Rio, N.~Ramsey, V.~Mattay, D.~Hommer,
  J.~Frank, D.~Weinberger, Statistical analysis of functional {MRI} data in the
  wavelet domain, Medical Imaging, IEEE Transactions on 17~(2) (1998) 142--154.

\bibitem{Ville2004}
D.~V.~D. Ville, T.~Blu, M.~Unser, Integrated wavelet processing and spatial
  statistical testing of f{MRI} data, Neuroimage 23~(4) (2004) 1472--85.

\bibitem{Hagmann2008}
P.~Hagmann, L.~Cammoun, X.~Gigandet, R.~Meuli, C.~J. Honey, V.~J. Wedeen,
  O.~Sporns, Mapping the structural core of human cerebral cortex, PLoS Biol
  6~(7) (2008) e159.

\bibitem{Wessel1996}
P.~Wessel, W.~H.~F. Smith, \href{www.ngdc.noaa.gov/mgg/shorelines/gshhs.html}{A
  global, self-consistent, hierarchical, high-resolution shoreline database}, J
  Geophys. Res. 101(B4) (1996) 8741--8743.
\newline\urlprefix\url{www.ngdc.noaa.gov/mgg/shorelines/gshhs.html}

\bibitem{Kondor2002}
R.~I. Kondor, J.~Lafferty, Diffusion kernels on graphs and other discrete input
  spaces, in: Proceedings of the 19th International Conference on Machine
  Learning, 2002.

\end{thebibliography}
\bibliographystyle{elsarticle-num}

\end{document}